\documentclass[11pt,english]{amsart}

\usepackage[a4paper]{geometry}
\usepackage{amssymb}
\usepackage{amsmath}
\usepackage{nccmath}
\usepackage{fancyhdr}
\usepackage{color}
\usepackage{enumerate}
\usepackage{tikz}
\newtheorem{theorem}{Theorem}[section] % 1st argument is your name for it
\newtheorem{lemma}[theorem]{Lemma}     % 2nd argument is what is printed
\newtheorem{corollary}[theorem]{Corollary}
\newtheorem{proposition}[theorem]{Proposition}
\usepackage{mathrsfs}
\usepackage{moreverb}
\usepackage{dsfont}

\newtheorem{example}[theorem]{Example}
\newtheorem{definition}[theorem]{Definition}

\newtheorem{remark}[theorem]{Remark}

\newtheorem*{theorem*}{Theorem}

\usepackage[bookmarksopen, naturalnames]{hyperref}

\makeatother

\begin{document}
	
	\title[Common frequently hypercyclic random vectors]{Common frequently hypercyclic random vectors}
	
	\author{A. Mouze and V. Munnier}
	\thanks{}
	\address{Augustin Mouze, Univ. Lille, \'Ecole Centrale de Lille, CNRS, UMR 8524 - Laboratoire Paul Painlev\'e  F-59000 Lille, France}
	\email{augustin.mouze@centralelille.fr}
	\address{Vincent Munnier, 16 avenue Pasteur, 94100 Saint Maur Des Foss\'es, France}
	\email{munniervincent@hotmail.fr}
	
	
	\begin{abstract} 
		We study common frequently hypercyclic vectors for countable families of weighted backward shifts acting on $\ell_p$ spaces, $1\leq p<\infty$. Using probabilistic techniques, we develop a general existence criterion, complemented by a non-existence result. These insights are then applied to the specific setting of countable families of polynomials of weighted backward shifts, providing conditions under which they share a common frequently hypercyclic vector.
	\end{abstract}
	%% maketitle must follow the abstract.
	\maketitle                   % Produces the title.
	
		\footnotetext{The first author is partially supported by the grant ANR-24-CE40-0892-01 of the French
		National Research Agency ANR (project ComOp).}

	\section{Introduction and notations.} 
Throughout this paper, $\mathbb{N}$ denotes the set of positive integers $\{1, 2, \dots\}$, while $\mathbb{N}_0 := \mathbb{N} \cup \{0\}$. For any finite set $E$, we denote its cardinality by $\# E$.

\begin{definition}
\rm{	The \textit{lower natural density} of a subset $A \subseteq \mathbb{N}$ is defined by:
	\[
	\underline{\hbox{dens}}(A) = \liminf\limits_{n \to \infty} \frac{\#(A \cap [1,n])}{n}
	\]
When the limit exists, we define the \textit{natural density} of $A$ by:
\[
\hbox{dens}(A) = \lim_{n \to \infty} \frac{\#(A \cap [1,n])}{n}
\].}
\end{definition}

We begin by recalling some notions of linear dynamics. Let $X$ be a separable topological vector space and $T:X\rightarrow X$ a continuous linear operator.
\begin{definition} \rm{A vector $x \in X$ is called a \textit{hypercyclic vector} for $T$ if its orbit 
$\{T^n x : n \in \mathbb{N}\}$ is dense in $X$. The operator $T$ is said to be 	
\textit{hypercyclic} if it possesses a hypercyclic vector. }
\end{definition}
\noindent In \cite{BayGriv}, Bayart and Grivaux introduced a stronger form of hypercyclicity. 

\begin{definition}
	\rm{A vector 
$x \in X$ is called a \textit{frequently hypercyclic vector} for $T$ if for 
each non-empty open set $U \subset X$ the set $\{n : T^n x \in U\}$ has positive lower density. The operator $T$ is said to be \textit{frequently hypercyclic}
if it possesses a frequently hypercyclic vector. }	
	\end{definition}
Thus, while hypercyclicity requires the orbit of a vector to visit every region of the space infinitely often, frequent hypercyclicity demands that the frequency of these visits has a positive lower density. This ensures that the orbit returns to every open set with a quantifiable regularity, implying that the vector spends a non-negligible proportion of time in every part of the space; ultimately, this bridges the gap between pure topological dynamics and ergodic theory.

The sets of hypercyclic and frequently hypercyclic vectors for the operator $T$ will be denoted by $HC(T)$ and $FHC(T)$, respectively. Such behaviors of vectors may seem unusual, but in the context of  hypercyclicity, they can be quite common. Indeed, the existence of hypercyclic vectors is often established via the Hypercyclicity Criterion, which ensures that $HC(T)$ is a dense $G_{\delta}$-set whenever it is non-empty. Consequently, the set $HC(T)$ is of the second Baire category. In contrast, proving the existence of frequently hypercyclic vectors - for instance, via the Frequent Hypercyclicity Criterion - is a constructive process that does not rely on the Baire Category Theorem. This is for good reason: the set $FHC(T)$ is always a set of Baire first category \cite{BayRuz,Mooth}. 
\\

A problem which has been extensively studied is that of common hypercyclicity. By an immediate consequence of the Baire Category Theorem, any countable family of hypercyclic operators on a Fr\'echet space shares a common hypercyclic vector; in fact, the set of such vectors is $G_{\delta}$-dense. The situation changes, however, when uncountable families of operators are considered. The first positive result in this direction was obtained by Abakumov and Gordon \cite{AG} who showed that the family $\{\lambda B:\ \vert \lambda\vert>1\}$ on $\ell_2$ - the Hilbert space of square-summable complex sequences - possesses  a dense $G_{\delta}$ set of common hypercyclic vectors, where $B$ denotes the backward shift operator. Subsequently, these results were generalized, notably through an elegant criterion established by Costakis and Sambarino \cite{CS}. For a comprehensive overview of linear dynamics and these developments, we refer the reader to the monographs \cite{BayMath} and  \cite{grossebook}. \\
By contrast, the property of common frequent hypercyclicity for families of operators has rarely been considered - largely due to the meager structure of the set $FHC(T)$  - until a recent paper  \cite{CEMM}, in which the authors establish criteria for such families to possess a shared frequently hypercyclic vector. This present paper contributes to this line of research by focusing on weighted backward shifts acting on $\ell_p$-spaces, $1\leq p<\infty$.\\ 

Let us fix $1\leq p<+\infty$. We recall that $\ell_p$ stands for the space of all sequences $x=(x_n)_{n\in \mathbb{N}_0}$ of scalars for which $\Vert x \Vert_p:=(\sum_{n\in \mathbb{N}_0}|x_n|^p)^{1/p} < +\infty$. Endowed with the norm $\Vert \cdot \Vert_p$, it is a Banach space. The unit sequence $(e_n)_{n\in \mathbb{N}_0}$ is a boundlessly complete unconditional basis of $\ell_p$. For a sequence $x = (x_n)$, its support is defined as the set of indices for which the terms of the sequence are non-zero. 
The set of all sequences with finite support is denoted by $c_{00}$. The elements of $c_{00}$ will also be referred to as polynomials. The subspace $c_{00}$ is dense in $\ell_p$.
For a non-zero polynomial $x=(x_n)_{n\in\mathbb{N}_0}$, its degree is defined as $\deg(x) = \max \{ n \in \mathbb{N}_0 : x_n \neq 0 \}$. \\
Let $\boldsymbol{w}=(w_k)_{k\in\mathbb{N}}$ be a \textit{weight} sequence, i.e. a bounded sequence of non-zero scalars. Let also $B_{\boldsymbol{w}}$ be the associated weighted backward shift on $\ell_p$ defined by $B_{\boldsymbol{w}}e_0:=0$ and $B_{\boldsymbol{w}}e_n:=w_ne_{n-1}$ for all $n\geq 1$. It is well-known that $B_{\boldsymbol{w}}$ is boundedness. When $\boldsymbol{w}$ is the constant sequence equal to $1$, the associated backward shift operator will be conventionally denoted by $B$. Weighted backward shifts play a crucial role in linear dynamics. They are instrumental in developing examples and counter-examples; indeed, the weight sequence provides a straightforward interpretation of the dynamical properties of the operators. We refer the reader to the books \cite{BayMath,grossebook}. For instance, $B_{\boldsymbol{w}}$ is hypercyclic if and only if $\limsup\limits_{n\rightarrow\infty}\vert w_1\dots w_n\vert=\infty$. In an important paper \cite{BayRuz}, Bayart and Rusza showed that the operator $B_{\boldsymbol{w}}$ is frequently hypercyclic if 
	and only if 
	\begin{equation}\label{condweighthyp}
		C_{\boldsymbol{w}}:=\left(\sum_{j\geq 1}\frac{1}{\vert w_1w_2\dots w_j\vert^p}\right)^{1/p}<+\infty.
	\end{equation} 

Returning to the problem of common frequent hypercyclicity, it is noteworthy that the criteria established in \cite{CEMM} show, in particular, that the operators $\lambda B$, $\lambda\in\Lambda\subset\mathbb{C}$ admit a common frequently hypercyclic vector if and only if the set $\{\vert\lambda\vert:\ \lambda\in\Lambda\}$ is a countable relatively compact subset of $(1,\infty)$. Analogous results are established for the existence or non-existence of common frequently hypercyclic vectors for families of the type $\lambda B_{\boldsymbol{w}}$. The proofs are very technical and follow a constructive approach. In a recent work, Grivaux, Matheron and Menet demonstrated that softer arguments based on measure theory and ergodicity are insufficient to prove the existence of a common hypercyclic vector, even in the case of two operators $aB$ and $bB$. Actually they showed that the operators $aB$ and $bB$ are in fact orthogonal: any two invariant measures $m_a$ and $m_b$ for $aB$ and $bB$ respectively with $m_a(\{0\})=m_b(\{0\})=0$ are necessarily orthogonal \cite{GMM}. 
The aim of this paper is to nevertheless provide \textit{a probabilistic construction of common frequently hypercyclic vectors }for finite or countable families of weighted backward shifts. 
The random construction of hypercyclic vectors has been the subject of numerous studies over the past ten years, specifically focusing on the growth of hypercyclic functions in settings like shift operators on the space of entire functions and on the space of holomorphic functions on the unit disk \cite{Agn, Agn2, MouMun, MouMun2, Niku}. Here our random construction is based, in part, on a slight modification of a random approximation result from \cite{MouMun}. This variant ensures that, given a frequently hypercyclic weighted backward shift $B_{\boldsymbol{w}}$, one can find a random vector $Z:=\sum_{j\geq 0}X_je_j$ in $\ell_p$ almost surely, satisfying the following approximation property: for any given $\varepsilon>0$ and $h\in c_{00}$, almost surely, there exists a subset $A\subseteq\mathbb{N}$ with positive natural density such that, for all $n\in A$, $\Vert \sum_{k=0}^nX_{k+n}w_{k+1}\ldots w_{k+n}e_k-h\Vert_p<\varepsilon$. We then construct a common frequently hypercyclic vector by suitably concatenating blocks of frequently hypercyclic random vectors corresponding to each operator in the family. This allows us to obtain the random version of the criterion of common frequent hypercyclicity for countable families of weighted shifts with geometric growth of weight products, as developed in \cite{CEMM}, which is the content of Theorem \ref{weightedshiftscounta}. It covers the case of the operators $\lambda B$. Furthermore, in the main section of this paper, by refining this method, we provide a sufficient condition for the existence of a common (random) frequently hypercyclic vector for families of weighted backward shifts with more general growth rates. This condition is expressed in terms of the summability of a series involving the ratio of weight products. This result is stated as Theorem \ref{weightedshiftscounta22}. Conversely, we provide sufficient conditions for the non-existence of common frequently hypercyclic vectors based on the non-summability of a similar series. We refer to Theorem \ref{interden}. We illustrate these theorems with new examples. Finally we show that these results can be extended to countable families of polynomials of weighted backward shift operators, by viewing these operators as shifts in adapted bases. This approach was previously used in \cite{MouMun}, and we rely on this same construction. The main result of this section is Theorem \ref{polyweightedshiftscounta22}, where sufficient conditions are provided under which a family of polynomials of frequently hypercyclic weighted backward shifts shares a common frequently hypercyclic vector. As corollary, we show that this result extends to polynomials of weighted shifts even if the shift operators themselves fail to be frequently hypercyclic.\\

The paper is organized as follows: Section \ref{prelim} presents preliminary results which will be employed throughout the remainder of this work, while Section \ref{geomcommon} focuses on the random version of the known criterion of common frequent hypercyclicity for countable families of weighted shifts with geometric growth of weight products. In Section \ref{abstheory}, we deal with the sufficient conditions of existence or non-existence of common frequently hypercyclic vectors for more general families of weighted backward shifts. Finally, in Section \ref{poly}, we are interested in families of polynomials of weighted backward shift operators. 

\vskip5mm

\noindent \textit{Notations.} Throughout the paper, $(\Omega,\mathcal{B},\mathbb{P})$ will be a standard probability space. We will say that the support of  a complex random variable $X$ is the {\it whole complex plane} if for every non-empty open set $U\subset\mathbb{C}$, we have $\mathbb{P}(X^{-1}(U))>0.$ 
Following the standard Landau notation, we will denote by $u_n=\mathcal{O}(v_n)$ if there exists a constant $C>0$ such that the inequality $\vert u_n\vert\leq C\vert v_n\vert$ holds for all $n$ sufficiently large.
Finally, whenever $A$ and $B$ depend on some parameters, we will use the notation $A\lesssim B$ for meaning $A\leq CB$ for some constant $C>0$ that does not depend on the involved parameters. 
%apart from $p$ and $\omega\in \Omega$. 

\section{Preliminary results}\label{prelim} In this section, we present several definitions and preliminary lemmas that will be central to our subsequent proofs. We begin by the following classical yet useful observation which will be essential for the probabilistic constructions.

\begin{lemma}\label{lemmefondamproba}
	For all $k\geq 1$, let $(Z_{k,n})_{n \in \mathbb{N}}$ be a sequence of random vectors in $\ell_p$ ($1 \le p < \infty$) and let $\underline{\mathcal{D}}$ be the family of subsets of $\mathbb{N}$ with positive lower density. Assume that for every $k\geq 1$, $h \in c_{00}$ and every $\varepsilon > 0$:
	\[
	\mathbb{P} \left( \{ \omega \in \Omega : \{ n \in \mathbb{N} : \| Z_{k,n}(\omega) - h \|_p < \varepsilon \} \in \underline{\mathcal{D}} \} \right) = 1.
	\]
	Then, for almost every $\omega \in \Omega$, the following property holds:
	\[
	\forall k\geq 1, \ \forall h \in \ell_p,\ \forall \varepsilon > 0,\ \{ n \in \mathbb{N} : \| Z_{k,n}(\omega) - h \|_p < \varepsilon \} \in \underline{\mathcal{D}}.
	\]
\end{lemma}

\begin{proof} First, we note that for any $h \in \ell_p$ and $\varepsilon > 0$, the set 
		\[
		E_{h,\varepsilon,k}=\{ \omega \in \Omega : \{ n \in \mathbb{N} : \| Z_{k,n}(\omega) - h \|_p < \varepsilon \} \in \underline{\mathcal{D}} \}
		\]
		is measurable. Indeed, since $\underline{\mathcal{D}}$ is defined by the positivity of the lower density, the condition $\omega\in E_{h,\varepsilon ,k}$ is equivalent to 
		\[ \liminf\limits_{N \to \infty} \frac{1}{N} \sum_{n=1}^N \mathds{1}_{ \{ \| Z_{k,n}(\omega) - h \|_p < \varepsilon \} } > 0,
	\]
		which involves only countable operations on measurable functions. 
		Since $1 \leq p < \infty$, the set $D$ of sequences in $c_{00}$ with coordinates in $\mathbb{Q}+i\mathbb{Q}$ is a countable dense subset in $\ell_p$. For any $k\geq 1$, $h \in D$ and $s \in \mathbb{N}^*$, define the event:
		\[
		E_{h, s,k} = \{ \omega \in \Omega : \{ n \in \mathbb{N} : \| Z_{k,n}(\omega) - h \|_p < 1/s \} \in \underline{\mathcal{D}} \}.
		\]
		By hypothesis, $\mathbb{P}(E_{h, s,k}) = 1$ for all $k\geq 1$, every $h \in D$ and $s \in \mathbb{N}^*$. Set
		\[
		E = \bigcap_{h \in D} \bigcap_{k\geq 1}\bigcap_{s \in \mathbb{N}^*} E_{h, s,k}.
		\]
		Since a countable intersection of sets of measure 1 has measure 1, we have $\mathbb{P}(E) = 1$. Now, let $\omega \in E$, $h' \in \ell_p$, $k\geq 1$ and $\varepsilon > 0$. By the density of $D$ in $\ell_p$, there exists $h \in D$ such that $\|h' - h\|_p < \varepsilon/2$. Choose $s \in \mathbb{N}^*$ such that $1/s < \varepsilon/2$. Since $\omega \in E_{h, s,k}$, we have $\{ n \in \mathbb{N} : \| Z_{k,n}(\omega) - h \|_p < 1/s \} \in \underline{\mathcal{D}}$. By the triangle inequality we get
		\[
		\|Z_{k,n}(\omega) - h'\|_p \leq \|Z_{k,n}(\omega) - h\|_p + \|h - h'\|_p < \frac{\varepsilon}{2} + \frac{\varepsilon}{2} = \varepsilon.
		\]
		We deduce 
		$$\{ n \in \mathbb{N} : \| Z_{k,n}(\omega) - h \|_p < 1/s \} \subseteq \{ n \in \mathbb{N} : \| Z_{k,n}(\omega) - h' \|_p < \varepsilon\} $$
	which gives $\{ n \in \mathbb{N} : \| Z_{k,n}(\omega) - h' \|_p < \varepsilon\}\in\underline{\mathcal{D}}$. The proof is finished.
	\end{proof}

Furthermore, for our purposes, it will be useful to partition a vector $v = \sum_{n \geq 0} v_n e_n $ in $\ell_p$ into blocks of specified size. The appropriate decomposition will be given by
\[
v=v_0e_0+\sum_{n=0}^{\infty}P_{n,m}(v)
\]
where
\[
P_{n,m} (v) := \sum_{j=m^n}^{m^{n+1}-1} v_j e_j
\]
and $m$ is a positive integer fixed according to the specific construction. 
We begin by showing that the intersection of the supports of a suitable selection of blocks with a set $A \subseteq \mathbb{N}$ of positive natural density yields a set with positive lower density. In order to build this particular selection, let us first introduce the $2$-adic valuation $v_2(n)$ of a positive integer by
$$v_2(n)=\max\{k\in\mathbb{N}_0:2^k\vert n\}.$$
Moreover, in the following, $\psi$ will denote the mapping
$$\begin{array}{rcl}\psi :\mathbb{N}&\rightarrow&\mathbb{N}\\n&\mapsto&v_2(n)+1.\end{array}$$
Finally, for all $k\geq 1$, let us define the subset $R_k$ of $\mathbb{N}$ as follows
$$R_k:=\{n\in\mathbb{N}\ ;\ \psi(n)=k\}.$$ 
It is easy to check that 
$$R_k=\{2^{k-1}(2j+1)\ ;\ j\in\mathbb{N}_0\}.$$
Thus the regularity of the distribution of the sets $R_k$ then yields the following lemma. 
	
	\begin{lemma}\label{lemma_combination_density}
		Let $A$ be a subset of $\mathbb{N}$ such that $\rm{dens}(A)>0$ and $m,\gamma$ be integers with $2\leq\gamma<m$. Then, for all $k\geq 1$, 
		$$ \underline{\rm{dens}}(A\cap \bigcup_{n\in R_k} [m^n,\ldots,\gamma m^{n}]\cap\mathbb{N})>0.$$
	\end{lemma}
	
	\begin{proof}
		Let us display $A$ as an increasing sequence $\displaystyle A=({a}_{j})$. Since $\gamma >1$, let us choose $\varepsilon >0$ small enough such that $\displaystyle \gamma >\frac{1+\varepsilon}{1-\varepsilon}$. The hypothesis $\rm{dens}(A)=c>0$ is equivalent to $\displaystyle {a}_{j}\underset{j\rightarrow \infty}{\sim}\frac{j}{c}$. Therefore, there exists a positive integer $j_{\varepsilon}$ such that, for all $j\geq j_{\varepsilon} $,
		$$\frac{1}{c}(1-\varepsilon) j\leq {a}_{j} \leq \frac{1}{c}(1+\varepsilon)j.$$
		First, we are going to evaluate the minimum number of elements of $A$ that belong to the set $[{m}^{n},\ldots,\gamma {m}^{n}]\cap\mathbb{N}$. Set 
		$$ I_{n}:= \left\{j\in \mathbb{N}\mbox{ };\mbox{ } \frac{1}{c}(1-\varepsilon)j \geq {m}^{n} \mbox{ and } \frac{1}{c}(1+\varepsilon)j \leq \gamma {m}^{n}\right\}$$
		and observe that, for all $j\in I_n$, $a_j\in [{m}^{n},\ldots,\gamma {m}^{n}]\cap\mathbb{N}$.
		The inclusion 
		$$\displaystyle I_{n}\supset \left[\frac{c{m}^{n}}{1-\varepsilon};\frac{c\gamma{m}^{n}}{1+\varepsilon}\right]\cap \mathbb{N}$$ 
		ensures that
		$$\#I_{n}\geq c{m}^{n}\left(\frac{\gamma}{1+\varepsilon}-\frac{1}{1-\varepsilon} \right),$$
		which implies that, for all $n$ large enough, there is at least $c{m}^{n}\left(\frac{\gamma }{1+\varepsilon}-\frac{1}{1-\varepsilon} \right)$ elements of $A$ between ${m}^{n}$ and $\gamma m^n$. 
		By combining this property with the observation that consecutive elements of $R_k$ differ by $2^k$, we obtain
		$$ \underline{\rm{dens}}(A\cap \bigcup_{n\in R_k} [m^n,\ldots,\gamma m^{n}]\cap\mathbb{N})\geq \lim_{n\rightarrow +\infty}\left(\frac{c{m}^{n}\left(\frac{\gamma }{1+\varepsilon}-\frac{1}{1-\varepsilon} \right)}{\gamma {m}^{n+2^k}}\right)=\frac{c\left(\frac{\gamma }{1+\varepsilon}-\frac{1}{1-\varepsilon} \right)}{\gamma{m}^{2^k}}>0,$$
		which gives the desired result.
	\end{proof}

\noindent Finally let us also introduce a specific class of polynomial sequences in $\ell_p$.

\begin{definition}
	\rm{A sequence of polynomials $(u_k)$ in $\ell_p$ is said to be \textit{admissible} if it satisfies the following three conditions:
		\begin{enumerate}
				\item $u_0 = e_0$;
				\item $\deg(u_k) = k$ for all $k \geq 1$;
				\item $\sum_{k=0}^{\infty} \|u_k\|_p < \infty$.
			\end{enumerate}}
\end{definition}
\noindent Note that the previous conditions ensure that every element $h$ in $c_{00}$ can be written as $h=\sum_{k=0}^{d}\tilde{h}_ku_k$. \\

\noindent Moreover, given an admissible sequence $\boldsymbol{u}=(u_k)$ of polynomials in $\ell_p$, we define the associated subspace $E_p[\boldsymbol{u}] \subset \ell_p$ by:
\[
E_p[\boldsymbol{u}] := \left\{ \sum_{k=0}^{\infty} x_k u_k : (x_k) \in \mathbb{C}^{\mathbb{N}_0}, \sum_{k=0}^{\infty} |x_k| \|u_k\|_p < \infty \right\}
\]
We shall also consider the associated backward shift operator $B{[\boldsymbol{u}]}$ on $E_p[\boldsymbol{u}]$ defined by $B{[\boldsymbol{u}]}(u_0)=0$ and $B{[\boldsymbol{u}]}(u_k)=u_{k-1}$, for $k\geq 1$. 
Let us also define the sequence of truncated backward shift operators $(T_{n}(B{[\boldsymbol{u}]}^n))$ with respect the family $\boldsymbol{u}=(u_k)$ as follows
$$\begin{array}{rccl}T_{n}(B{[\boldsymbol{u}]}^n)\ :& E_p[\boldsymbol{u}]&\rightarrow &\ell_p\\&
		\sum_{k\geq 0}X_ku_k&\mapsto &\sum_{k=0}^{n}X_{k+n}u_k.\end{array}$$ 

\section{Common frequently hypercyclic vectors for a countable family of weighted backward shift operators}\label{geomcommon} In this section, we provide a random version of the common frequent hypercyclicity criterion for general families of weighted backward shifts \cite[Theorem 2.17]{CEMM}. This construction is based on the following result, which arises from a closer inspection of the proof of Theorem 4.3 in \cite{MouMun}. To state it, we first introduce the following specific decay property for a complex random variable $X$ supported on the entire complex plane $\mathbb{C}$:
\begin{equation}\label{NikulaDecayCond}
	\hbox{ for some }\beta>0,\quad 
	\limsup\limits_{r\rightarrow \infty} \left((\log r)^{1+\beta}\mathbb{P}(\vert X\vert\geq r)\right)<+\infty.
\end{equation}

\begin{theorem}\label{gdnbre} Let $1\leq p<\infty$. Let $(u_k)$ be an admissible sequence of polynomials in $\ell_p$ such that there exist $C>0$ and $0<\tau<1$ satisfying: for all $k\in\mathbb{N}$, $\Vert u_k\Vert_p\leq C \tau^k$. Let $X$ be a complex random variable whose support is the whole complex plane satisfying the decay condition (\ref{NikulaDecayCond}). Let $(X_k)$ be a sequence of 
	independent copies of $X.$ Then, for almost all realizations of the sequence $(X_n)$, the element $v=\sum_{k\geq 0}X_ku_k$ belongs to $E_p[\boldsymbol{u}]\subset\ell_p$. Moreover, for every $h$ in $c_{00}$ and $\eta >0$, almost surely, there exists a realization-dependent subset $A\subseteq\mathbb{N}$ with ${\rm{dens}}(A)>0$ such that, for all $n\in A$, $\Vert T_{n}(B{[\boldsymbol{u}]}^n)(v)-h\Vert_p<\eta$. 
\end{theorem}
	
\begin{proof} This result follows from the proof of Theorem 4.4 of \cite{MouMun}. We keep the notations of \cite{MouMun}. 
	Let us define $\varepsilon_n:=C(1+\varepsilon)^n\tau^{n},$ with $\varepsilon>0$ satisfying $\tau(1+\varepsilon)<1$ so that $\sum\varepsilon_n<\infty$. Moreover we get $\frac{\varepsilon_k}{\Vert u_k\Vert_p}\geq (1+\varepsilon)^k$ and the condition (\ref{NikulaDecayCond}) ensures that $\sum_{k\geq 0}\mathbb{P}(\vert X\vert\geq \frac{\varepsilon_k}{\Vert u_k\Vert_p})<\infty$. According to Lemma 4.2 of \cite{MouMun}, for almost all realizations of the sequence $(X_n)$ the random vector $v=\sum_{n\geq 0}X_nu_n$ belongs to $E_p[\boldsymbol{u}]\subset \ell_p$. Let $h$ be in $c_{00}$. We write $h=\sum_{n=0}^dh_nu_n$. Set, for all $k\geq d+1$,  
	$$B_{k,\eta}(h):=\left\{\sum_{j=0}^d\vert X_{j+k}-h_j\vert\Vert u_j\Vert_{p}<\eta/2
	\hbox{ and }\vert X_{j+k}\vert\leq \rho_j,\ j=d+1,\dots,k \right\},$$
	where $\rho_j=\delta\frac{\varepsilon_j}{\Vert u_j\Vert_{p}},$ with $0<\delta <1$ such that $\sum_{j\geq 0}\rho_j\Vert u_j\Vert_{p}=\delta\sum_{j\geq 0}\varepsilon_j<\eta/2.$ Now consider 
	$$p_{d+1}:=\mathbb{P}\left(\sum_{j=0}^d\vert X_{j+k}(\omega)-h_j\vert\Vert u_j\Vert_{p}\leq\eta/2
	\right)
	\hbox{ and }q_{d+1}:=\prod_{j\geq d+1}\mathbb{P}(\vert X\vert\leq \rho_j).$$
	Note that $p_{d+1}q_{d+1}>0$ because the distribution of $X$ has full support on the complex plane and the series $\sum_{j\geq d+1}\mathbb{P}(\vert X\vert >\rho_j)$ converges by condition (\ref{NikulaDecayCond}). Let us define the random variable
	$$Y_k=\frac{1}{k+1}\sum_{j=0}^k \mathds{1}_{B_{j,\eta}(h)}.$$ According to the proof of Theorem 4.4 of \cite{MouMun} 
	\[
	Y_k\rightarrow p_{d+1}q_{d+1}\hbox{ almost surely as } k\to\infty .
	\]
	This yields the desired result. Indeed, for a fixed realization $\omega$, we observe that 
	\[	
	Y_{k}(\omega) = \frac{\# \{j \leq k : \omega\in B_{j,\eta}(h)\}}{k+1}.
	\]
	By triangle inequality if $\omega\in B_{j,\eta}(h)$, then  $\Vert T_{j}(B{[\boldsymbol{u}]}^j)(v)-h\Vert_p<\eta$. It follows that, almost surely, there exists a subset $A\subseteq \mathbb{N}$ with ${\rm{dens}}(A)>0$ such that, for all $n\in A$, $\Vert T_{n}(B{[\boldsymbol{u}]}^n)(v)-h\Vert_p<\eta$. 
\end{proof}

We will also require the following basic fact, which allows us to estimate the asymptotic behavior of a sequence of independent and identically distributed (i.i.d.) random variables in $L^p$. 

\begin{lemma}\label{estim_var} Let $1\leq p<\infty$. Let $(X_k)$ be a sequence of i.i.d. random variables in $L^p$. Then, for all $\varepsilon>0$, there exists a positive integer $k_{\varepsilon}$ such that, for all $k\geq k_{\varepsilon}$, $\vert X_k\vert\leq k^{\frac{1}{p}+\varepsilon}$ almost surely.  
\end{lemma}

\begin{proof} Using Markov inequality, we get, for all positive integers $k$, 
	$$\mathbb{P}(\vert X_k\vert^p\geq k^{1+p\varepsilon})\leq \frac{\mathbb{E}[\vert X_k\vert^p]}{k^{1+p\varepsilon}}.$$	
	Thus the series $\sum_{k\geq 0}\mathbb{P}(\vert X_k\vert^p\geq k^{1+p\varepsilon})$ converges and Borel-Cantelli lemma ensures that 
	$$\mathbb{P}\left(\limsup \left(\vert X_k\vert\geq k^{\frac{1}{p}+\varepsilon}\right)\right)=0.$$ 	
\end{proof}

\noindent Now we are ready to prove the following result. 

	\begin{theorem}\label{weightedshiftscounta} Let $1\leq p<\infty$ and $\boldsymbol{w(i)}=(w_n(i))_{n\in\mathbb{N}}$, $i\in \mathbb{N}$, be countably many weights for which every weighted backward shift $B_{\boldsymbol{w(i)}}$, $i\in \mathbb{N}$, is a continuous operator on the complex Banach space $\ell_p$. We assume that there exist a weight $\omega=(\omega_n)_{n\in\mathbb{N}}$, constants $M> 1$ and $0<\eta < 1$ and a constant $C>0$, such that for any $i\in \mathbb{N}$ and any $n\geq 0$, $m\geq 1$,
		\begin{enumerate}[(i)]
			\item \label{lablab1}the series $\sum_{k \geq 1}\vert\omega_1\ldots\omega_{k}\vert^{-p}$ is convergent;
			\item \label{lablab2}$\vert \omega_n\ldots \omega_{n+m} \vert \leq C\eta^m \vert w_n(i)\ldots w_{n+m}(i) \vert$;
			\item \label{lablab3}$C^{-1}M ^{-m}\leq \vert w_n(i)\ldots w_{n+m}(i) \vert \leq CM^m$.
		\end{enumerate}
		Let $(X_j)$ be a sequence of i.i.d. complex random variables in $L^p$. Then there exists a positive integer $m$ such that the vector $Z=\sum\limits_{j\geq m}\frac{X_j}{w_1\left(\psi(\lfloor\frac{\log (j)}{\log(m)}\rfloor)\right)\dots w_j\left(\psi(\lfloor\frac{\log (j)}{\log(m)}\rfloor)\right)}e_j $ is almost surely a common frequently hypercyclic vector for the family $(B_{\boldsymbol{w(i)}})_{i\in \mathbb{N}}$.
	\end{theorem}
	
	\begin{proof} We define $C_{\boldsymbol{\omega}}$ as $C_{\boldsymbol{\omega}}:=(\sum\limits_{k \geq 1}\vert\omega_1\ldots\omega_{k}\vert^{-p})^{1/p}$. Let us first define 
	$${Z}_k:=\sum_{n\geq 1}\frac{X_n}{w_1(k)\dots w_n(k)}e_n=\sum_{n\geq 0} P_{n,m}(Z_k).$$	
	Observe that 
	$$Z:=\sum_{n\geq 1}\sum_{j=m^n}^{m^{n+1}-1}\frac{X_j}{w_1(\psi(n))\ldots w_j(\psi(n))}e_j=\sum_{n\geq 1}P_{n,m}(Z_{\psi(n)}).$$
	The vector $Z$ is thus composed of the blocks of the vectors $Z_k$ arranged in a specific order. We shall first show that these blocks inherit the frequent approximation property established in Theorem \ref{gdnbre}. For all $k\geq 1$, we set the sequence $\boldsymbol{u^{(k)}}=(u^{(k)}_j)$, such that $u_{0}^{(k)}=e_0$ and, for all $j\geq 1$, $u_{j}^{(k)}=\frac{e_j}{w_1(k)\ldots w_j(k)}$. Clearly we have, for all $k\geq 1$, $B[\boldsymbol{u^{(k)}}]=B_{\boldsymbol{w(k)}}$ and 
	$$\Vert u_{j}^{(k)}\Vert_p\leq C\eta^{-1}C_{\boldsymbol{\omega}}\eta^j.$$
	Now let $k\geq 1$, $h$ be in $c_{00}$ and $\varepsilon>0$. Since the random variables $X_n$ belong to $L^p$, they satisfy the decay condition (\ref{NikulaDecayCond}) by Markov's inequality. According to Theorem \ref{gdnbre}, almost surely, there exists a subset $A_k$ of $\mathbb{N}$ with $\rm{dens}(A_k)>0$ such that 
	\begin{equation}\label{equpfmain1}
		\Vert T_{n}(B_{\boldsymbol{w(k)}}^n)({Z_k})-h\Vert_p<\varepsilon.
	\end{equation}	
		Let $\gamma\geq 2$ be an integer. Let us choose $m\in\mathbb{N}$ such that 
	\begin{equation}\label{estimation_bloc0}
		m\geq\max\left(2\gamma, \frac{\gamma(4\log (M)-\log(\eta))}{\vert\log(\eta)\vert}\right)+1.
	\end{equation}
For all $n\in R_k$ and for every $l\in [m^n,\gamma m^n]\cap\mathbb{N}$, we get, since $2\gamma m^n<m^{n+1}$,
\begin{align}\label{equpfmain4}
	\begin{split}
		B_{\boldsymbol{w(k)}}^l(Z)=&\sum_{j=l}^{2l}\frac{X_j}{w_1(k)\ldots w_{j-l}(k)}e_{j-l} +\sum_{j=2l+1}^{m^{n+1}-1}\frac{X_j}{w_1(k)\ldots w_{j-l}(k)}e_{j-l}\\&+
		\sum_{j\geq m^{n+1}}X_j\frac{ w_{j-l+1}(k)\ldots w_j(k)}{w_1\left(\psi(\lfloor\frac{\log (j)}{\log(m)}\rfloor)\right)\ldots w_{j}\left(\psi(\lfloor\frac{\log (j)}{\log(m)}\rfloor)\right)}e_{j-l}\\:=&\ U_{n,l,1}+U_{n,l,2}+U_{n,l,3}.
	\end{split}
\end{align}
Our goal is to estimate $\Vert B_{\boldsymbol{w(k)}}^l(Z)-h\Vert_p$, making use of the following properties: for all $n\in R_k$ and for any $l\in [m^n,\gamma m^n]\cap\mathbb{N}$, 
	\begin{equation}\label{equpfmain2}
	\forall j\geq l,\quad \left\vert\frac{1}{w_1(k)\ldots w_{j-l}(k)}\right\vert=\left\vert\frac{\omega_1\ldots\omega_{j-l}}{w_1(k)\ldots w_{j-l}(k)}\frac{1}{\omega_1\ldots\omega_{j-l}}\right\vert\leq C \frac{\eta^{j-l-1}}{\omega_1\ldots\omega_{j-l}}
	\end{equation}	
and 
\begin{align}\label{equpfmain3}
				\begin{split}\displaystyle\forall i\in\mathbb{N},\ \forall j\geq m^{n+1},\quad \left\vert\frac{w_{j-l+1}(k)\ldots w_j(k)}{ w_1(i)\ldots w_j(i)}\right\vert&=\displaystyle
			\left\vert\frac{w_{j-l+1}(k)\ldots w_j(k)}{w_{j-l+1}(i)\ldots w_j(i)}\frac{\omega_1\ldots\omega_{j-l}}{w_1(i)\ldots w_{j-l}(i)}\frac{1}{\omega_1\ldots\omega_{j-l}}\right\vert\\&
			\leq\displaystyle C^3M^{2(l-1)}\eta^{j-l-1} \frac{1}{\vert \omega_1\ldots\omega_{j-l}\vert}.
			\end{split}
		\end{align}

We begin by considering the term $\Vert U_{n,l,2}\Vert_p$. In view of Lemma \ref{estim_var}, the inequality $(\ref{equpfmain2})$ and the fact that the sequence $(j^{p\varepsilon+1}\eta^{p(j-l-1)/2})_{j\geq l}$ is bounded, we get, almost surely, for all $n\in R_k$ large enough and $l\in [m^n,\gamma m^n]\cap\mathbb{N}$,
\begin{equation}\label{equpfmain5}
	\begin{array}{rcl}\Vert U_{n,l,2}\Vert_p^p &\leq&\displaystyle \sum_{j=2l+1}^{m^{n+1}-1}\frac{\vert X_j\vert^p }{\vert w_1(k)\ldots w_{j-l}(k)\vert^p}\\
		&\lesssim&\displaystyle\sum_{j=2l+1}^{m^{n+1}-1}\frac{j^{1+p\varepsilon}\eta^{p(j-l-1)}}{\vert \omega_1\ldots \omega_{j-l}\vert^p}\\
	&\lesssim&\displaystyle \eta^{pl/2}\sum_{j=2l+1}^{m^{n+1}-1}\frac{1}{\vert \omega_1\ldots \omega_{j-l}\vert^p}.
	\end{array}
\end{equation}
This implies that, almost surely, for all $n\in R_k$ large enough and $l\in [m^n,\gamma m^n]\cap\mathbb{N}$,
\begin{equation}\label{equpfmain6}
	\Vert U_{n,l,2}\Vert_p\lesssim \eta^{m^n/2}.
\end{equation}
The term $U_{n,l,3}$ is estimated using similar arguments. Indeed by Lemma \ref{estim_var}, (\ref{equpfmain3}) and recalling that the sequence $(j^{p\varepsilon+1}\eta^{p(j-l-1)/2})_{j\geq l}$ is bounded, we have, almost surely, for all $n\in R_k$ large enough and $l\in [m^n,\gamma m^n]\cap\mathbb{N}$,
\begin{equation}\label{equpfmain7}
	\begin{array}{rcl}\Vert U_{n,l,3}\Vert_p^p &\leq&\displaystyle \sum_{j\geq m^{n+1}}\vert X_j\vert^p\frac{ \vert w_{j-l+1}(k)\ldots w_j(k)\vert^p}{\left\vert w_1\left(\psi(\lfloor\frac{\log (j)}{\log(m)}\rfloor)\right)\ldots w_{j}\left(\psi(\lfloor\frac{\log (j)}{\log(m)}\rfloor)\right)\right\vert^p}\\
		&\lesssim&\displaystyle \sum_{j\geq m^{n+1}}M^{2p(l-1)}\frac{j^{1+p\varepsilon}\eta^{p(j-l-1)}}{\vert \omega_1\ldots \omega_{j-l}\vert^p}
		\\&\lesssim&\displaystyle \left(M^{2\gamma}\eta^{(m-\gamma)/2}\right)^{p m^n}\sum_{j\geq m^{n+1}}\frac{1}{\vert \omega_1\ldots \omega_{j-l}\vert^p}.
	\end{array}
\end{equation}
We deduce that, almost surely, for all $n\in R_k$ large enough and $l\in [m^n,\gamma m^n]\cap\mathbb{N}$,
\begin{equation}\label{equpfmain8}
	\Vert U_{n,l,3}\Vert_p\lesssim (M^{2\gamma}\eta^{(m-\gamma)/2})^{m^n},
\end{equation}
where $0<M^{2\gamma}\eta^{(m-\gamma)/2}<1$ thanks to (\ref{estimation_bloc0}).
Finally, observe that for all $n\in R_k$ large enough and $l\in [m^n,\gamma m^n]\cap\mathbb{N}$, 
\begin{equation}
	\sum_{j=l}^{2l}\frac{X_j}{w_1(k)\ldots w_{j-l}(k)}e_{j-l}-h=\sum_{j=0}^{l}\frac{X_{j+l}}{w_1(k)\ldots w_{j}(k)}e_{j}-h= T_{l}(B_{\boldsymbol{w(k)}}^l)({Z}_k)-h.
\end{equation}
Combining (\ref{equpfmain4}) with (\ref{equpfmain1}), (\ref{equpfmain6}) and (\ref{equpfmain8}), we find that, almost surely, there exist a positive integer $N_k$ and a subset $A_k\subseteq\mathbb{N}$ with $\hbox{dens}(A_k)>0$ such that for all $n\in A_k\cap \bigcup\limits_{j\in R_k; \atop j\geq N_k} [m^j,\ldots,\gamma m^{j}]\cap\mathbb{N}$ 
\[
\Vert B_{\boldsymbol{w(k)}}^n(Z)-h\Vert_p<3\varepsilon.
\]
Since Lemma \ref{lemma_combination_density} guarantees that $\underline{\hbox{dens}}(A_k\cap \bigcup\limits_{n\in R_k;\atop n\geq N_k} [m^n,\ldots,\gamma m^{n}]\cap\mathbb{N}) >0$, Lemma \ref{lemmefondamproba} allows us to conclude that the vector $Z$ is almost surely a common frequently hypercyclic vector for the family $(B_{\boldsymbol{w(i)}})_{i\in \mathbb{N}}$.	
	\end{proof}

\begin{remark}
	\rm{The reader may notice that the assumptions $(i)$, $(ii)$ and $(iii)$ of Theorem \ref{weightedshiftscounta} are identical to those required by the criterion provided by Theorem 2.17 of \cite{CEMM}. In this sense, the result is not new. However, the construction of a common frequently hypercyclic random vector is new, and provides an alternative approach to establishing the existence of common frequently hypercyclic vectors. Here, the result applies to weighted shifts whose the product of weights exhibits geometric growth. It therefore applies here, in particular, to a family of backward shifts $\{\lambda B,\ \lambda\in\Lambda\}$ provided that the set $\{\vert\lambda\vert:\ \lambda\in\Lambda\}$ is a countable, relatively compact subset of $(1,\infty)$.}
\end{remark}
	
\section{A general randomized criterion}\label{abstheory} In this section, we generalize Theorem \ref{weightedshiftscounta} by establishing a general criterion for the existence of common random frequently hypercyclic vectors for weighted backward shifts on $\ell_p$, by relaxing the hypotheses $(i)$, $(ii)$ and $(ii)$ of that theorem to allow for slower growth of the weight products. We shall see that the resulting statement generalizes Remark 2.19 in \cite{CEMM}, which established the existence of common frequently hypercyclic vectors for weighted backward shifts, provided the quotients of the weight products remain uniformly bounded from both above and below.\\

We start by introducing some notations. Let $(\alpha_k)_{k \in \mathbb{N}_0}$ be an increasing sequence of positive real numbers such that $\alpha_k \to\infty$ as $k\to\infty$ and 
	\begin{equation}\label{condsuite} 
	\hbox{ for some } \tilde {\varepsilon}>0,\quad \ \sum_{l=0}^k\alpha_l=\mathcal{O}\left(\frac{k^2}{(\log(k))^{1+\tilde{\varepsilon}}}\right).
	\end{equation} 
Condition (\ref{condsuite}) ensures that the sequence $\left(\frac{\log(n)\alpha_n}{n}\right)$ tends to $0$. Indeed, since $(\alpha_k)$ is increasing, we have
\[
n \alpha_n\leq\sum_{k=n+1}^{2n}\alpha_k\lesssim \frac{n^2}{(\log(n))^{1+\tilde{\varepsilon}}}.
\]
Consequently, we obtain $\frac{\log(n)\alpha_n}{n}\lesssim (\log(n))^{-\tilde{\varepsilon}}\rightarrow 0$, as $n\to\infty$.
\vskip3mm

\noindent \textit{Notations.} For the sake of clarity, throughout this section, given a weight sequence $(w_n)$, we denote its partial products by $W_n:=w_1\ldots w_n$, ($n=1,2,\dots$) with the convention $W_0:=1$. \\

A more general version of Theorem \ref{gdnbre} is now presented, assuming better integrability for the random variables. This result shall facilitate the generalization of Theorem \ref{weightedshiftscounta} to cases involving slower growth of the weight products.
		 
	\begin{theorem}\label{gdnbre2} Let $1\leq p<\infty$. Let $\boldsymbol{w}=(w_n)_{n\in\mathbb{N}}$ be a sequence of weights satisfying: there exist $C,\tau>0$ such that $\vert W_n\vert^{-p}\leq C n^{-1}[\log(n+1)]^{-(1+\frac{p}{2}+\tau)}$. Let $(\alpha_k)$ be a strictly increasing sequence of positive real numbers such that $\alpha_k \to\infty$ as $k\to\infty$ satisfying condition (\ref{condsuite}). Let $X$ be a Gaussian complex random variable such that the support of the distribution of $X$ is whole $\mathbb{C}$ and $(X_n)$ be a sequence of independent copies of $X.$ For almost every $\omega\in\Omega$, the element $v=\sum_{k\geq 1}X_k W_k^{-1}e_k$ belongs to $\ell_p$. Moreover, for every $h$ in $c_{00}$ and $\eta >0$, almost surely, there exists a realization-dependent subset $A\subseteq \mathbb{N}$ with ${\rm{dens}}(A)>0$ such that, for all $n\in A$, $\Vert T_{\lfloor\alpha_n\rfloor}(B_{\boldsymbol{w}}^n)(v)-h\Vert_p< \eta$.
	\end{theorem}

	\begin{proof} First, since 
		\[ \mathbb{E}[\Vert v\Vert_p]\leq \mathbb{E}[\vert X\vert^p]\sum_{n\geq 1}\vert W_n\vert^{-p}<\infty,
		\]
		the random vector $\sum_{n\geq 1}X_nW_n^{-1} e_n$ belongs to $\ell_p$ almost surely.\\ 
Let $\eta>0$ and $h=\sum_{n=0}^dh_nu_n$ be in $c_{00}$. Let us define $\varepsilon_n=C n^{-1}\left[\log (n+1)\right]^{-(1+\frac{\tau}{2})}$ so that $\sum_{n\geq 1}\varepsilon_n<\infty$.  We set $\rho_n=\delta\varepsilon_n\vert W_n\vert^p,$ where we have taken $0<\delta <1$ so that 
\[
\sum_{n\geq 1}\rho_n\vert W_n\vert^{-p}<\eta/2.
\] 
It follows that $\rho_n=\delta \varepsilon_n \vert W_n\vert^p\geq \delta \left[\log (n+1)\right]^{\frac{p}{2}+\frac{\tau}{2}}$. Since $X$ is a Gaussian complex random variable, we deduce that there exists $C_1>0$ such that, for all $n\geq 1$,  
		$$\mathbb{P}\left(\vert X_n\vert^p\geq \rho_n\right)\leq e^{-C_1 (\delta \varepsilon_n \vert W_n\vert ^p)^{2/p}}\leq e^{-C_1 \delta^{\frac{2}{p}}[\log(n+1)]^{1+\frac{\tau}{p}}}.$$
		We deduce $\sum\limits_{n\geq 1} \mathbb{P}\left(\vert X\vert^p\geq \rho_n\right)<\infty$. Set 
$$B_{k,\eta,m}(h)=\left\{\sum_{j=0}^d\vert X_{j+k}-h_j\vert^p \vert W_j\vert^{-p}<\eta/2\hbox{ and }\vert X_{j+k}\vert^p\leq \rho_j,\ j=d+1,\dots,m \right\},$$ 
with $m\geq d+1$. 
%and $\rho_j=\delta\varepsilon_j\vert W_j\vert^p,$ where we have taken $0<\delta <1$ so that $\sum_{j\geq 0}\rho_j\vert W_j\vert^{-p}<\eta/2.$ 
Since the sequence of random variables $(X_n)$ is independent, it follows that, for all $k\geq d,$
$$\mathbb{P}(B_{k,\eta,m}(h))=p_{d+1}\frac{q_{d+1}}{q_{m+1}},$$
with 
$$p_{d+1}=\mathbb{P}\left(\left\{\sum_{j=0}^d\vert X_{j+k}(\omega)-h_j\vert^p\vert W_j\vert^{-p}\leq\eta/2\right\}\right)\hbox{ and }q_k=\prod_{j\geq k}\mathbb{P}(\vert X\vert^p\leq \rho_j).$$
Since the series $\sum_{j\geq 1} \mathbb{P}(\vert X\vert^p >\rho_j)$ converges, the sequence $(q_k)$ tends to $1$ as $k$ tends to $\infty$. Let us define the random variable
$$Y_k=\frac{1}{k+1}\sum_{j=0}^k \mathbf{1}_{B_{j,\eta,\lfloor\alpha_j\rfloor}(h)}.$$ 
We get 
$$\mathbb{E}[Y_k]=\frac{p_{d+1}}{k+1}\sum_{j=0}^k\frac{q_{d+1}}{q_{\lfloor\alpha_j\rfloor +1}}$$ 
and Ces\`aro's Theorem ensures that 
		\begin{equation}\label{cesaro}
			\mathbb{E}[Y_k]\rightarrow p_{d+1}q_{d+1},\quad \hbox{as } k\rightarrow\infty.
		\end{equation} 
		Moreover the following equality holds
		\begin{equation}\label{variance1}
			\begin{array}{rcl}(k+1)^2\hbox{var}[Y_k]&=&\displaystyle\sum_{l=0}^k\mathbb{P}(B_{l,\eta,\lfloor\alpha_l\rfloor}(h))+2\sum_{0\leq l<l'\leq k}\mathbb{P}(B_{l,\eta,\lfloor\alpha_l\rfloor}(h)\cap B_{l',\eta,\lfloor\alpha_{l'}\rfloor}(h))\\&&\displaystyle-{p_{d+1}^2}q_{d+1}^2\left(\sum_{j=0}^k\frac{q_{d+1}}{q_{\lfloor\alpha_j\rfloor+1}}\right)^2.\end{array}
		\end{equation}
		Observe that 
		\[\hbox{for }l+\lfloor\alpha_l\rfloor <l', \quad B_{l,\eta,\lfloor\alpha_l\rfloor}(h)\hbox{ and } B_{l',\eta,\lfloor\alpha_{l'}\rfloor}(h)\hbox{ are independent,}\]
		and ,
		\[\hbox{for }l+d<l'\leq l+\lfloor\alpha_l\rfloor,\quad  
		B_{l,\eta,\lfloor\alpha_l\rfloor}(h)\cap B_{l',\eta,\lfloor\alpha_{l'}\rfloor}(h)\subset B_{l,\eta,l'-l-1}(h)\cap B_{l',\eta,\lfloor\alpha_{l'}\rfloor}(h).\] 
		By independence, we get, for all $0\leq l<l'\leq k,$
		\begin{equation}\label{estim1}\begin{array}{rcll}
				\mathbb{P}(B_{l,\eta,\lfloor\alpha_l\rfloor}\cap B_{l',\eta,\lfloor\alpha_{l'}\rfloor})&=&\displaystyle\frac{p_{d+1}^2q_{d+1}^2}{q_{1+\lfloor\alpha_l\rfloor}q_{1+\lfloor\alpha_{l'}\rfloor}}
				&\hbox{ for }l+\lfloor\alpha_l\rfloor< l',\\&&&\\
				\mathbb{P}(B_{l,\eta,\lfloor\alpha_l\rfloor}\cap B_{l',\eta,\lfloor\alpha_{l'}\rfloor})
				&\leq&\displaystyle\frac{p_{d+1}^2q_{d+1}^2}{q_{l'-l}q_{1+\lfloor\alpha_{l'}\rfloor}}&\hbox{ for }d+l<l'\leq l+\lfloor\alpha_l\rfloor.
			\end{array}
		\end{equation}
Combining the equality $\sum\limits_{0\leq l<l'\leq k}=\sum\limits_{l=0}^{k-1}\left(\sum\limits_{l'=l+1}^{l+d}+\sum\limits_{l'=l+d+1}^{l+\lfloor\alpha_l\rfloor}+\sum\limits_{l'=l+\lfloor\alpha_l\rfloor+1}^{k}\right)$ with (\ref{variance1}) and (\ref{estim1}), we get, for all $l$ large enough, 
			\begin{equation}\label{variance2}
			\begin{array}{rcl}(k+1)^2\hbox{var}[Y_k]&\leq& \displaystyle \sum_{l=0}^k\frac{p_{d+1}q_{d+1}}{q_{\lfloor\alpha_l\rfloor +1}}+2\sum_{l=0}^{k-1}\sum_{l'=l+1}^{l+d}\mathbb{P}(B_{l,\eta,\lfloor\alpha_l\rfloor}(h)\cap B_{l',\eta,\lfloor\alpha_{l'}\rfloor}(h))\\&&\displaystyle +2\sum\limits_{l=0}^{k-1}\sum\limits_{l'=l+d+1}^{l+\lfloor\alpha_l\rfloor}\frac{p_{d+1}^2q_{d+1}^2}{q_{l'-l}q_{1+\lfloor\alpha_{l'}\rfloor}}+2\sum\limits_{l=0}^{k-1}\sum\limits_{l'=l+\lfloor\alpha_l\rfloor+1}^{k}\frac{p_{d+1}^2q_{d+1}^2}{q_{1+\lfloor\alpha_l\rfloor}q_{1+\lfloor\alpha_{l'}\rfloor}}\\&&\displaystyle-p_{d+1}^2q_{d+1}^2\left(\sum_{l=0}^k\frac{1}{q_{\lfloor\alpha_l\rfloor +1}}\right)^2\end{array}.
		\end{equation}	
Taking into account the equality 
		\[
		\left(\sum_{l=0}^k\frac{1}{q_{\lfloor\alpha_l\rfloor +1}}\right)^2=\sum_{l=0}^k\left(\frac{1}{q_{\lfloor\alpha_l\rfloor +1}}\right)^2+2\sum_{0\leq l<l'\leq k}\frac{1}{q_{1+\lfloor\alpha_l\rfloor}q_{1+\lfloor\alpha_{l'}\rfloor}},
		\]	
we obtain
		\begin{equation}\label{variance3}
			\begin{array}{rcl}(k+1)^2\hbox{var}[Y_k]&\leq& \displaystyle \sum_{l=0}^k\frac{p_{d+1}q_{d+1}}{q_{\lfloor\alpha_l\rfloor +1}}+2\sum_{l=0}^{k-1}\sum_{l'=l+1}^{l+d}\mathbb{P}(B_{l,\eta,\lfloor\alpha_l\rfloor}(h)\cap B_{l',\eta,\lfloor\alpha_{l'}\rfloor}(h))\\&&\displaystyle +2p_{d+1}^2q_{d+1}^2\sum\limits_{l=0}^{k-1}\sum\limits_{l'=l+d+1}^{l+\lfloor\alpha_l\rfloor}\frac{1}{q_{1+\lfloor\alpha_{l'}\rfloor}}\left( \frac{1}{q_{l'-l}}-  \frac{1}{q_{1+\lfloor\alpha_l\rfloor}}\right)\\&&\displaystyle-p_{d+1}^2q_{d+1}^2\sum_{l=0}^k\left(\frac{1}{q_{\lfloor\alpha_l\rfloor +1}}\right)^2\\&\leq&
				\displaystyle (k+1)+2dk+2\sum_{l=0}^{k-1}(\lfloor\alpha_l\rfloor -d). \end{array}
		\end{equation}		
We deduce
		\[
		\hbox{var}[Y_k]\leq \frac{1}{k}+\frac{2}{k^2}\sum_{l=0}^{k-1}\lfloor\alpha_l\rfloor=\mathcal{O}\left(\frac{1}{(\log(k))^{1+\tilde{\varepsilon}}}\right).
		\]
Let $\lambda>0$ and $0<\tilde{\delta}<\tilde{\varepsilon}$. Let us consider the subsequence $(\phi(k))_{k\geq 1}$ defined by, for all $k\geq 1$, $\phi(k)=\lfloor \hbox{exp}(k^{\frac{1+\tilde{\delta}}{1+\tilde{\varepsilon}}})\rfloor.$ Combining (\ref{cesaro}) with Chebyshev inequality and (\ref{variance3}), we get, for all $k$ large enough,
		$$\begin{array}{rcl}\displaystyle\mathbb{P}(\vert Y_{\phi(k)}-p_{d+1}q_{d+1}\vert >\lambda)&\leq&\displaystyle  
			\mathbb{P}(\vert Y_{\phi(k)}-\mathbb{E}(Y_{\phi(k)})\vert >\lambda/2)+
			\mathbb{P}(\vert \mathbb{E}(Y_{\phi(k)})-p_{d+1}q_{d+1}\vert >\lambda/2)\\&\leq&\displaystyle 
			\mathcal{O}\left(\frac{1}{\lambda^2 k^{1+\tilde{\delta}}}\right).\end{array}$$
Hence Borel-Cantelli Lemma ensures that
		$$Y_{\phi(k)}\to p_{d+1}q_{d+1}\hbox{ almost surely as }k\to  
		\infty.$$
Since the support of the distribution of $X$ is the whole complex plane, we have $p_{d+1}q_{d+1}>0.$ Then, for all $k\in \mathbb{N},$ there exists a positive integer 
		$n_k$ such that $\phi(n_k)\leq k<\phi(n_k+1)$. Thus we get
		\begin{equation}\label{fin}
			Y_{\phi(n_k)}\frac{\phi(n_k)}{\phi(n_k+1)}\leq Y_k\leq Y_{\phi(n_k+1)}\frac{\phi(n_k+1)}{\phi(n_k)}.
		\end{equation}
		Since $\phi(n_k)=\lfloor \hbox{exp}(n_k^{\frac{1+\tilde{\delta}}{1+\tilde{\varepsilon}}})\rfloor$ with $0<\tilde{\delta}<\tilde{\varepsilon}$, we deduce that $\frac{\phi(n_k+1)}{\phi(n_k)}$ tends to $1$. Therefore the inequality (\ref{fin}) leads to
		$$Y_k\to p_{d+1}q_{d+1}\hbox{ almost surely as } k\to \infty.$$ 
		The desired result follows by finishing the argument as in the end of the proof of Theorem \ref{gdnbre}. 
	\end{proof}

Let $1\leq p<\infty $ and $\boldsymbol{w(k)}=(w_n(k))_{n\in\mathbb{N}}$, $k\in \mathbb{N}$, be countably many weights for which every weighted backward shift $B_{\boldsymbol{w(k)}}$, $k\in \mathbb{N}$, is a continuous operator on the complex Banach space $\ell_p$. Under appropriate conditions, Theorem \ref{gdnbre2} allows us to construct a random vector that serves as a common frequently hypercyclic vector for the family of operators $B_{\boldsymbol{w(k)}}$, $k\in \mathbb{N}$. First, to apply Theorem \ref{gdnbre2}, we assume that for each $k\geq 1$ there exist constants $C_k,\tau_k>0$ such that $\vert W_n(k)\vert^{-p}\leq C_k n^{-1}[\log(n+1)]^{-(1+\frac{p}{2}+\tau_k)}$. Moreover, we suppose that, for all $j\in\mathbb{N}$, $\inf_{k\in\mathbb{N}} \vert W_j(k)\vert>0$. Let us define, for every $k\geq 1$, 
\[
{Z}_k:=\sum_{n\geq m}\frac{X_n}{W_n(k)}e_n\quad \hbox{ and }\quad Z:=\sum\limits_{j\geq m}\frac{X_j}{W_j\left(\psi(\lfloor\frac{\log (j)}{\log(m)}\rfloor)\right)}e_j=\sum_{n\geq 1}P_{n,m}(Z_{\psi(n)})
\]
where $m\geq 2$ is a positive integer to be specified later. Let also $(\alpha_j)$ be a strictly increasing sequence of positive real numbers such that $\alpha_j \to\infty$ as $j\to\infty$ satisfying condition (\ref{condsuite}).
\\

Now let $k\geq 1$, $h$ be in $c_{00}$ and $\varepsilon>0$. According to Theorem \ref{gdnbre2}, there exists a subset $A_k$ of $\mathbb{N}$ with $\rm{dens}(A_k)>0$ such that 
\begin{equation}\label{lastequpfmain1}
	\Vert T_{\lfloor\alpha_n\rfloor}(B_{\boldsymbol{w(k)}}^n)({Z_k})-h\Vert_p<\varepsilon\quad\hbox{ almost surely}.
\end{equation}	
Let $\gamma\geq 2$ be an integer. Fix a positive integer $m\geq \gamma+1$ to be specified later. For all $n\in R_k$ and for every $l\in [m^n,\gamma m^n]\cap\mathbb{N}$ large enough, we get
\begin{align}\label{equabstract}
	\begin{split}
		B_{\boldsymbol{w(k)}}^l(Z)=&\sum_{j=l}^{l+\lfloor\alpha_l\rfloor}\frac{X_j}{W_{j-l}(k)}e_{j-l} +\sum_{j=l+\lfloor\alpha_l\rfloor+1}^{m^{n+1}-1}\frac{X_j}{W_{j-l}(k)}e_{j-l}\\&+
		\sum_{j\geq m^{n+1}}\frac{X_j}{W_{j-l}(k)}\frac{ W_j(k)}{W_{j}\left(\psi(\lfloor\frac{\log (j)}{\log(m)}\rfloor)\right)}e_{j-l}\\:=&\ \sum_{j=l}^{l+\lfloor\alpha_l\rfloor}\frac{X_j}{W_{j-l}(k)}e_{j-l}+Q(n,l,k).
	\end{split}
\end{align}
Since $\log(k)\alpha_k/k\to 0$, we have for every $l\in [m^n,\gamma m^n]\cap\mathbb{N}$ large enough 
$$ l+\lfloor\alpha_l\rfloor+1\leq \gamma m^n+\lfloor\alpha_{\gamma m^n}\rfloor+1< (\gamma+1)m^n\leq m^{n+1}.$$ 
Moreover, observe that, for all $l\in [m^n,\gamma m^n]\cap\mathbb{N}$,
	\begin{align}\label{equabstract2}
		\begin{split}
			\mathbb{E}\left[\Vert Q(n,l,k)\Vert_p^p\right]\lesssim &\sum_{j=\lfloor\alpha_l\rfloor+1}^{m^{n+1}-l-1}\frac{1}{\vert W_{j}(k)\vert^p}+\sum_{j\geq m^{n+1}}\left\vert\frac{W_j(k)}{W_{j}\left(\psi(\lfloor\frac{\log (j)}{\log(m)}\rfloor)\right)}\right\vert^p\frac{1}{\vert W_{j-l}(k)\vert^p}
		\end{split}
	\end{align}
	We set 
	$$R_n{(k,l)}:=\sum_{j=\lfloor\alpha_l\rfloor+1}^{m^{n+1}-l-1}\frac{1}{\vert W_{j}(k)\vert^p}+\sum_{j\geq m^{n+1}}\left\vert\frac{W_j(k)}{W_{j-l}(k)\inf_{i}\vert W_{j}\left(i\right)\vert}\right\vert^p$$
	and
	$$R_n(k):=\sup\left(R_n(k,l)\ ;\ l\in [m^n,\gamma m^n]\cap\mathbb{N}\right).$$
Thus, we get, for all $l\in [m^n,\gamma m^n]\cap\mathbb{N}$,
$$\mathbb{E}\left[\Vert Q(n,l,k)\Vert_p^p\right]\lesssim  R_n(k).$$ 
Now, assume that
\begin{equation}\label{hypotheseabstractfonda} 
	\forall k\geq 1,\quad \sum\limits_{n=1}^{\infty} R_n(k) < \infty.
	\end{equation} 
Markov's inequality implies
	\[ \mathbb{P}(\Vert Q(n,l,k)\Vert_p \geq \varepsilon) \leq \frac{\mathbb{E}[\Vert Q(n,l,k)\Vert_p^p]}{\varepsilon^p}, 
	\]
	which yields:
	\[ \sum_{n=1}^{\infty} \mathbb{P}(\Vert Q(n,l,k)\Vert_p \geq \varepsilon) < \infty. \]
	Then according to Borel-Cantelli Lemma, the event $\{\Vert Q(n,l,k)\Vert_p\geq \varepsilon\}$ occurs for only finitely many $n$ almost surely. We conclude, for $n$ sufficiently large,
	\begin{equation} \label{abstracteqq2}
		\Vert Q(n,l,k)\Vert_p<\varepsilon \quad \text{almost surely.} 
	\end{equation}
Finally observe that, for all $n$ large enough, $l\in [m^n,\gamma m^n]\cap\mathbb{N}$, 
\begin{equation}
	\sum_{j=l}^{l+\lfloor\alpha_l\rfloor}\frac{X_j}{W_{j-l}(k)}e_{j-l}-h=\sum_{j=0}^{\lfloor\alpha_l\rfloor}\frac{X_{j+l}}{W_{j}(k)}e_{j}-h= T_{\lfloor\alpha_l\rfloor}(B_{\boldsymbol{w(k)}}^l)({Z}_k)-h.
\end{equation}
Combining (\ref{lastequpfmain1}) with (\ref{abstracteqq2}), we obtain that there exist a positive integer $N_k$ and a subset $A_k$ of $\mathbb{N}$ with $\hbox{dens}(A_k)>0$ such that for all $n\in A_k\cap \bigcup\limits_{n\in R_k;\atop n\geq N_k} [m^n,\ldots,\gamma m^{n}]\cap\mathbb{N}$ 
\[
\Vert B_{\boldsymbol{w(k)}}^n(Z)-h\Vert_p<2\varepsilon\quad\hbox{ almost surely.}
\]
Since Lemma \ref{lemma_combination_density} guarantees that $\underline{\hbox{dens}}(A_k\cap \bigcup\limits_{n\in R_k;\atop n\geq N_k} [m^n,\ldots,\gamma m^{n}]\cap\mathbb{N}) >0$, Lemma \ref{lemmefondamproba} allows us to conclude that the random vector $Z$ is almost surely a common frequently hypercyclic vector for the family $(B_{\boldsymbol{w(k)}})_{k\in \mathbb{N}}$.
\\

In light of the above, it suffices to verify the fundamental hypothesis given by (\ref{hypotheseabstractfonda}) to obtain a common frequently hypercyclic random vectors for the family of operators $(B_{\boldsymbol{w(k)}})_{k\in\mathbb{N}}$. This leads to the following statement:

\begin{theorem}\label{weightedshiftscounta22} Let $1\leq p<\infty$ and $\boldsymbol{w(k)}=(w_n(k))_{n\in\mathbb{N}}$, $k\in \mathbb{N}$, be countably many weights for which every weighted backward shift $B_{\boldsymbol{w(k)}}$, $k\in \mathbb{N}$, is a continuous operator on the complex Banach space $\ell_p$. Assume that, for all $n\in\mathbb{N}$, $\inf_{k\in\mathbb{N}} \vert W_n(k)\vert>0$, and, for every $k\in\mathbb{N}$, there exist $C_k>0$ and $\tau_k>0$ such that, $\vert W_n(k)\vert^{-p}\leq C_k n^{-1}[\log(n+1)]^{-(1+\frac{p}{2}+\tau_k)}$. If there exist positive integers $2\leq \gamma<m$ and an increasing sequence $(\alpha_k)$ of positive real numbers such that $\alpha_k \to\infty$ satisfying condition (\ref{condsuite}) such that 
\[
\forall k\geq 1,\quad \sum_{n\geq 1}\sup\left(R_n(k,l)\ ;\ l\in [m^n,\gamma m^n]\cap\mathbb{N}\right)<\infty
\]
where $R_n(k,l)=\sum_{j=\lfloor\alpha_l\rfloor+1}^{m^{n+1}-l-1}\frac{1}{\vert W_{j}(k)\vert^p}+\sum_{j\geq m^{n+1}}\left\vert\frac{W_j(k)}{W_{j-l}(k)\inf_i\vert W_{j}\left(i\right)\vert}\right\vert^p$, then there exists a common frequently hypercyclic random vector for the family $(B_{\boldsymbol{w(k)}})_{k\in \mathbb{N}}$.
\end{theorem}

We can therefore derive the following corollaries.

\begin{corollary}\label{weightedshiftscounta23} Let $1\leq p<\infty$. Let $\frac{1}{p}<a$ and $(\beta_k)_{k\in\mathbb{N}}$ a sequence of real numbers in $[a,\infty)$. Let $(\boldsymbol{w(k)})_{k\in\mathbb{N}}$ a sequence of weights satisfying: there exists $C\geq 1$ such that, for every $k\geq 1$, $C^{-1}n^{\beta_k}\leq \vert w_1(k)\ldots w_n(k)
	\vert\leq C n^{\beta_k}$. Then there exists a common frequently hypercyclic random vector for the family $(B_{\boldsymbol{w(k)}})_{k\in \mathbb{N}}$ on $\ell_p$.
\end{corollary}

\begin{proof} Clearly, we have, for all $k\geq 1$, $\vert W_n(k)\vert^{-p} \leq C^p n^{-p\beta_k}\leq C^p n^{-pa}$. Thus the assumption of Theorems \ref{gdnbre2} is satisfied. Moreover, observe that  $\inf_{k\in\mathbb{N}}\vert W_n(k)\vert \geq C^{-1}n^a>0$. Let us choose $\sigma>1/(pa-1)$ and $(\alpha_l)$ the sequence given by $\alpha_l=(\log(l))^{\sigma}$. It is easy to check that the sequence $(\alpha_l)$ satisfies condition (\ref{condsuite}). Moreover, for every $k\geq 1$, for $l\in [m^n,\gamma m^n]$ large enough,  
	\[
	\sum_{j\geq \lfloor\alpha_l\rfloor+1}\frac{C^p}{\vert W_{j}(k)\vert^p}\leq \sum_{j\geq \lfloor\alpha_l\rfloor+1}\frac{C^p}{j^{\beta_k p}}\leq \frac{C^p}{p\beta_k-1}\frac{1}{(\lfloor\alpha_l\rfloor +1)^{p\beta_k-1}}\leq \frac{C^p}{p\beta_k-1}\frac{1}{(n^{\sigma}\log^\sigma(m))^{pa-1}}
	\]
and 
\[
\begin{array}{rcl}\displaystyle\sum_{j\geq m^{n+1}}\left\vert\frac{W_j(k)}{W_{j-l}(k)\inf_i\vert W_{j}\left(i\right)\vert}\right\vert^p&\leq& C^{3p}\displaystyle\sum_{j\geq m^{n+1}}\frac{(1-\frac{l}{j})^{-p\beta_k}}{j^{pa}}\\ &\leq& C^{3p}\displaystyle
(1-\frac{\gamma}{m})^{-p\beta_k}\sum_{j\geq m^{n+1}}j^{-pa}\\
&\leq&\displaystyle C^{3p}\frac{(1-\frac{\gamma}{m})^{-p\beta_k}}{pa-1}\frac{1}{(m^{n+1}-1)^{pa-1}}.\end{array}
\]
Since $\sigma(pa-1)>1$, the hypotheses of Theorem \ref{weightedshiftscounta22} are satisfied. The conclusion holds. 
\end{proof}

\noindent Theorem \ref{weightedshiftscounta22} also recovers the fact that if $\Lambda\subset(1,\infty)$ is countable and relatively compact, the family $\{\lambda B;\ \lambda\in\Lambda\}$ shares  common frequently hypercyclic vector.

\begin{corollary}\label{weightedshiftscounta24}Let $1\leq p<\infty$. Let $1<a<b$ and $(\lambda_k)_{k\in\mathbb{N}}$ a sequence of real numbers in $[a,b]$. Let $(\boldsymbol{w(k)})_{k\in\mathbb{N}}$ a sequence of weights sequence such that, for every $k\geq 1$, $w_n(k)=\lambda_k$. Then there exists a common frequently hypercyclic random vector for the family $(B_{\boldsymbol{w(k)}})_{k\in \mathbb{N}}$ on $\ell_p$.
\end{corollary}

\begin{proof} 
	Clearly for all $k\geq 1$, $\vert W_n(k)\vert^{-p} =\lambda_k^{-np}\leq a^{-np}\lesssim n^{2}$ and thus the assumption of Theorem \ref{gdnbre2} is satisfied. Moreover, observe that  $\inf_{k\in\mathbb{N}}\vert W_n(k)\vert \geq a^n>0$. Let $\gamma\geq 2$, $m\geq \max(2\gamma+1,\frac{\gamma\log(b)}{\log(a)}+1)$ and $(\alpha_l)$ be the sequence given by $\alpha_l=\log(l)$. Clearly this sequence $(\alpha_l)$ satisfies condition (\ref{condsuite}). Moreover, for every $k\geq 1$, for $l\in [m^n,\gamma m^n]$ large enough,  
	\[
	\sum_{j\geq \lfloor\alpha_l\rfloor+1}\frac{1}{\vert W_{j}(k)\vert^p}= \sum_{j\geq \lfloor\alpha_l\rfloor+1}\lambda_k^{-jp}\leq \lambda_{k}^{-p\alpha_l}\leq \lambda_k^{-pn\log(m)} 
	\]
	and 
	\[
	\begin{array}{rcl}\displaystyle\sum_{j\geq m^{n+1}}\left\vert\frac{W_j(k)}{W_{j-l}(k)\inf_i\vert W_{j}\left(i\right)\vert}\right\vert^p&\leq& \displaystyle\sum_{j\geq m^{n+1}}\frac{\lambda_k^{pl}}{a^{jp}}\\ &\leq& \displaystyle
		\frac{a}{a-1}\ e^{p m^n(\gamma\log(\lambda_k)-m\log(a))}\\
		&\leq&\displaystyle \frac{a}{a-1}\ e^{-p m^n\log(a)}.\end{array}
	\]
	Since the series $\sum_{n\geq 1}\lambda_k^{-pn\log(m)} $ and  $\sum_{n\geq 1}e^{-p m^n\log(a)} $ are convergent, the hypotheses of Theorem \ref{weightedshiftscounta22} are satisfied. The conclusion holds. 	
\end{proof}

\begin{remark} {\rm Condition (\ref{condsuite}) requires, in a sense, that the quotients of the weight products are comparable. This condition is natural, as shown by the following result.}
\end{remark}

\begin{proposition}\label{ctrex2} Let $1\leq p<\infty$. Let $\boldsymbol{v}=(v_n)_{n\in\mathbb{N}}$ and $\boldsymbol{w}=(w_n)_{n\in\mathbb{N}}$ be weight sequences such that the associated weighted backward shifts are frequently hypercyclic and hypercyclic, respectively, on $\ell_p$. Assume that, for all increasing sequences $(m_k)$ and $(n_k)$ of positive integers tending to infinity, with $\liminf\limits_{k\to\infty}(\frac{n_k}{m_k})>0$ and $\limsup\limits_{k\to\infty}(\frac{n_k}{m_k})<1$,  
	%for all $k\geq 1$, $n_k< m_k \leq b n_k$ for some positive integer $b>1$, 
	$$ \sum_{l\geq 1}\frac{\vert w_1\ldots w_{m_l}\vert^p}{\vert v_1\ldots v_{m_l} \vert^p\vert w_1\ldots w_{m_l-n_l}\vert^p}=\infty.$$
	Then $FHC(B_{\boldsymbol{v}})\cap HC(B_{\boldsymbol{w}})=\emptyset$.
\end{proposition}

\begin{proof} Assume that there exists $x=(x_j)_{j\in\mathbb{N}_0}\in\ell_p$ which is frequently hypercyclic for $B_{\boldsymbol{v}}$ and hypercyclic for $B_{\boldsymbol{w}}$. Let us write 
	\[
	x=x_0e_0+\sum_{k\geq 1}\frac{x_k}{v_1\ldots v_k}e_k.
	\]
	Denote by $\mathcal{N}_{v}$ the set
	\[
	\mathcal{N}_{\boldsymbol{v}}:=\{n\in\mathbb{N}\ :\ \Vert B_{\boldsymbol{v}}^n(x)-e_0\Vert_p< 1/2\}.
	\]
	By assumption, we have $\underline{\hbox{dens}}(\mathcal{N}_{v}):=c>0$. Let us choose a positive integer $m$ so that $m>2/c$. We claim that there exists $N\in\mathbb{N}$ such that, for all $n\geq N$, $\mathcal{N}_{\boldsymbol{v}}\cap [m^n,m^{n+1})\ne\emptyset$. Indeed, there exists $N_c\in\mathbb{N}$ so that, for all $n\geq N_c$,
	\[
	\frac{\#( \mathcal{N}_{\boldsymbol{v}}\cap [1,m^{n+1}))}{m^{n+1}}>c/2.
	\]
	If, for $n\geq N_c$, $\mathcal{N}_{\boldsymbol{v}}\cap [m^n,m^{n+1})=\emptyset$, then 
	\[
	\frac{c}{2}\leq \frac{\#( \mathcal{N}_{\boldsymbol{v}}\cap [1,m^{n+1}))}{m^{n+1}}=\frac{\#( \mathcal{N}_{\boldsymbol{v}}\cap [1,m^{n}))}{m^{n+1}}\leq \frac{1}{m}
	\]
	which gives a contradiction. \\
	Now, let us choose a sequence $(\varepsilon_l)$ of positive real numbers such that $\sum_{l\geq 1}\varepsilon_l^p<\infty$. For all $l\geq 1$, denote by $\mathcal{N}_l$ the set 
	\[
	\mathcal{N}_l:=\{n\in\mathbb{N}\ :\ \Vert B_{\boldsymbol{w}}^n(x)\Vert_p<\varepsilon_l\}.
	\]
	These sets are infinite. We construct by induction increasing sequences of positive integers $(n_l)$, $(p_l)$ and $(m_l)$ as follows: let $n_1\in \mathcal{N}_1\cap (m^{N_c},\infty)$ and let $p_1$ be the integer such that $m^{p_1}\leq n_1<m^{1+p_1}$. We then choose $m_1\in\mathcal{N}_{\boldsymbol{v}}\cap[m^{2+p_1},m^{3+p_1})$. Once $n_{l-1}$, $p_{l-1}$ and $m_{l-1}$ have been constructed, we choose 
	\[ 
	n_l\in \mathcal{N}_l\cap [m^{3+p_{l-1}},\infty)
	\]
	and $p_l$ the integer satisfying $m^{p_l}\leq n_l<m^{1+p_l}$, and we select 
	\[
	m_l\in\mathcal{N}_{\boldsymbol{v}}\cap[m^{2+p_l},m^{3+p_l}).
	\]
	Observe that 
	\begin{equation}
		m^{p_l}\leq n_l<m^{1+p_l}<m^{2+p_l}\leq m_l<m^{3+p_l}\leq m^3 n_l
	\end{equation}
	which ensures $\frac{1}{m^3}\leq\frac{n_l}{m_l}\leq \frac{1}{m}$. Therefore we have
	\begin{equation}\label{sequ222}
		\liminf\limits_{l\to\infty}\left(\frac{n_l}{m_l}\right)\geq \frac{1}{m^3}>0\quad \hbox{ and }\quad\limsup\limits_{l\to\infty}\left(\frac{n_l}{m_l}\right)\leq\frac{1}{m}<1.
	\end{equation}
	Since $(m_l)\subset \mathcal{N}_{\boldsymbol{v}}$, we get, for all $l\geq 1$,
	\[
	\vert x_{m_l}-1\vert^p+\sum_{k\geq m_l+1}\frac{\vert x_k\vert^p}{\vert v_1\ldots v_{k-m_l}\vert^p}<\frac{1}{2^p},
	\]
	which gives
	\begin{equation}\label{sequ223}
		\vert x_{m_l}\vert^p>1/2^p
	\end{equation}
	On the other hand, since $n_l\in \mathcal{N}_{l}$, we get, for all $l\geq 1$,
	\[
	\sum_{k\geq n_l}\frac{\vert x_k\vert^p}{\vert v_1\ldots v_k\vert^p}\frac{\vert w_1\ldots w_k  \vert^p}{\vert w_1\ldots w_{k-n_l} \vert^p}<\varepsilon_l^p.
	\]
	We deduce
	\[
	\frac{\vert x_{m_l}\vert^p}{\vert v_1\ldots v_{m_l}\vert^p}\frac{\vert w_1\ldots w_{m_l}  \vert^p}{\vert w_1\ldots w_{m_l-n_l} \vert^p}<\varepsilon_l^p
	\]
	which, together with (\ref{sequ223}), yields
	\begin{equation}\label{seq23}
		\frac{1}{\vert v_1\ldots v_{m_l}\vert^p}\frac{\vert w_1\ldots w_{m_l}  \vert^p}{\vert w_1\ldots w_{m_l-n_l} \vert^p}<2^p\varepsilon_l^p.
	\end{equation}
	By (\ref{sequ222}) and (\ref{seq23}), we have found two increasing sequences $(m_l)$ and $(n_l)$ such that $\liminf\limits_{l\to\infty}(\frac{n_l}{m_l})>0$, $\limsup\limits_{l\to\infty}(\frac{n_l}{m_l})<1$ and  
	\[
	\sum_{l\geq 1} \frac{1}{\vert v_1\ldots v_{m_l}\vert^p}\frac{\vert w_1\ldots w_{m_l}  \vert^p}{\vert w_1\ldots w_{m_l-n_l} \vert^p}\hbox{ is convergent.}
	\]
	This finishes the proof. 
\end{proof}

\begin{example} \rm {Let $1\leq p<\infty$. Let $\varepsilon >0$ and $\lambda>1$. Let us consider the sequences $\boldsymbol{w}=(w_n)_{n\in\mathbb{N}}$ and $\boldsymbol{v}=(v_n)_{n\in\mathbb{N}}$ given by, for $n\geq 1$, $v_n=\left(\frac{n+1}{n}\right)^{\frac{1+\varepsilon}{p}}$ and $w_n=\lambda$. The associated weighted backward shifts $B_{\boldsymbol{v}}$ and $B_{\boldsymbol{w}}=\lambda B$ are frequently hypercyclic on $\ell_p$ and the following property holds: $FHC(B_{\boldsymbol{v}})\cap FHC(\lambda B)=\emptyset$. Indeed, let $(m_j)$ and $(n_j)$ be increasing sequences of positive integers such that $cm_j\leq n_j\leq\delta m_j$ for some $0<c<\delta<1$. We get 
		\[
		\frac{\vert w_{m_j-n_j+1}\ldots w_{m_j}\vert}{\vert v_1\ldots v_{m_j}\vert}=\frac{\vert  \lambda\vert^{n_j}}{{m_j}^{\frac{1+\varepsilon}{p}}}\geq \frac{\vert  \lambda\vert^{cm_j}}{m_j^{\frac{1+\varepsilon}{p}}}
		\]
		which gives
		\[
		\frac{\vert w_{m_j-n_j+1}\ldots w_{m_j}\vert}{\vert v_1\ldots v_{m_j}\vert}\rightarrow\infty\hbox{ as }j\to\infty.
		\]
		Proposition \ref{ctrex2} allows us to conclude.}
\end{example}

	\begin{remark}\label{remimp}
		{\rm Let $1\leq p<\infty$. Let $\boldsymbol{v}=(v_n)_{n\in\mathbb{N}}$ and $\boldsymbol{w}=(w_n)_{n\in\mathbb{N}}$ be sequences such that the associated weighted backward shifts are frequently hypercyclic and hypercyclic, respectively, on $\ell_p$. The proof of Proposition \ref{ctrex2} shows that if the set $FHC(B_{\boldsymbol{v}})\cap HC(B_{\boldsymbol{w}})$ is non empty, then for any arbitrarily small $0<\varepsilon<1$, there exist increasing sequences $(m_k)$ and $(n_k)$ of positive integers tending to infinity, with $\liminf\limits_{k\to\infty}(\frac{n_k}{m_k})>0$ and $\limsup\limits_{k\to\infty}(\frac{n_k}{m_k})<\varepsilon$, such that
			$$ \sum_{l\geq 1}\frac{\vert w_1\ldots w_{m_l}\vert^p}{\vert v_1\ldots v_{m_l} \vert^p\vert w_1\ldots w_{m_l-n_l}\vert^p}<\infty.$$
			To see this, it suffices to choose $m>1/\varepsilon$ in the proof of Proposition \ref{ctrex2} following (\ref{sequ222}). 
		}
	\end{remark}
\noindent The previous remark leads to the following statement. 

\begin{theorem}\label{interfin} Let $1\leq p<\infty$. Let $\boldsymbol{v}=(v_n)_{n\in\mathbb{N}}$ and $\boldsymbol{w}=(w_n)_{n\in\mathbb{N}}$ be weight sequences such that the associated weighted backward shifts are frequently hypercyclic on $\ell_p$. If $FHC(B_{\boldsymbol{v}})\cap FHC(B_{\boldsymbol{w}})\ne\emptyset$, then, for any arbitrarily small $0<\varepsilon<1$, there exist increasing sequences $(m_k)$, $(n_k)$, $(m'_k)$, $(n'_k)$ of positive integers tending to infinity, with $\liminf\limits_{k\to\infty}(\frac{n_k}{m_k})>0$, $\limsup\limits_{k\to\infty}(\frac{n_k}{m_k})<\varepsilon$, $\liminf\limits_{k\to\infty}(\frac{n'_k}{m'_k})>0$ and  $\limsup\limits_{k\to\infty}(\frac{n'_k}{m'_k})<\varepsilon$, such that
	$$\sum_{l\geq 1}\frac{\vert w_1\ldots w_{m_l}\vert^p}{\vert v_1\ldots v_{m_l} \vert^p\vert w_1\ldots w_{m_l-n_l}\vert^p}<\infty\quad\hbox{ and }\quad\sum_{l\geq 1}\frac{\vert v_1\ldots v_{m'_l}\vert^p}{\vert w_1\ldots w_{m'_l} \vert^p\vert v_1\ldots v_{m'_l-n'_l}\vert^p}<\infty. $$
\end{theorem}	
\begin{proof}
	It suffices to note that the property $FHC(B_{\boldsymbol{v}})\cap FHC(B_{\boldsymbol{w}})\ne\emptyset$ implies that both $FHC(B_{\boldsymbol{v}})\cap HC(B_{\boldsymbol{w}})$ and $HC(B_{\boldsymbol{v}})\cap FHC(B_{\boldsymbol{w}})$ are non-empty, and then to apply Remark \ref{remimp}.
\end{proof}	
	
\noindent A countable version of this result can thus also be written. This statement should be compared with Theorem \ref{weightedshiftscounta22}. Following our previous notation, we set, for all $n\geq 1$, $W_n=w_1\ldots w_n$ for a weight sequence $(w_n)_{n\in\mathbb{N}}$.

 \begin{theorem}\label{interden} Let $1\leq p<\infty$ and $\boldsymbol{w(k)}=(w_n(k))_{n\in\mathbb{N}}$, $k\in \mathbb{N}$, be countably many weights for which every weighted backward shift $B_{\boldsymbol{w(k)}}$, $k\in \mathbb{N}$, is a continuous operator frequently hypercyclic on the complex Banach space $\ell_p$. Assume that $\bigcap_{k\geq 1}FHC(B_{\boldsymbol{w(k)}})\ne\emptyset$. Then then for all $i,j\geq 1$, $i\ne j$, for any arbitrarily small $0<\varepsilon<1$, there exist increasing sequences $(m_{i,j,l})_{l\in\mathbb{N}}$, $(n_{i,j,l})_{l\in\mathbb{N}}$ of positive integers tending to infinity, with $\liminf\limits_{l\to\infty}(\frac{n_{i,j,l}}{m_{i,j,l}})>0$ and $\limsup\limits_{l\to\infty}(\frac{n_{i,j,l}}{m_{i,j,l}})<\varepsilon$, such that
$$\sum_{l\geq 1}\frac{\vert W_{m_{i,j,l}}(i)\vert^p}{\vert W_{m_{i,j,l}}(j) \vert^p\vert W_{m_{i,j,l}-n_{i,j,l}}(i)\vert^p}<\infty.$$
 \end{theorem}		
	
The following example shows that Theorem \ref{weightedshiftscounta22} and Theorems \ref{interfin} or \ref{interden} jointly enable the connection of limit cases.

\begin{example} \rm {Let $0<\varepsilon<1$ and $\beta>1$. Let us consider sequences $\boldsymbol{w}=(w_n)$ and $\boldsymbol{v}=(v_n)$ so that, for $n\geq 2$, $w_1\ldots w_n=e^{\log^{1+\varepsilon}(n)}$ and $v_1\ldots v_n=n^{\beta}$. The associated weighted backward shifts $B_{\boldsymbol{v}}$ and $B_{\boldsymbol{w}}$ are frequently hypercyclic on $\ell_1$. We have 
		\[
		FHC(B_{\boldsymbol{w}})\cap HC(B_{\boldsymbol{v}})=\emptyset,\hbox{ if }\varepsilon>1,\quad\hbox{and}\quad FHC(B_{\boldsymbol{w}})\cap FHC(B_{\boldsymbol{v}})\ne\emptyset,\hbox{ if }0<\varepsilon\leq 1.
		\]
		 
		\begin{proof} First we consider the case $\varepsilon>1$. Let $(m_j)$ and $(n_j)$ be increasing sequences of positive integers such that $cm_j\leq n_j\leq\delta m_j$ for some $0<c<\delta<1$. We get 
			\[
			\frac{\vert w_{1}\ldots w_{m_j}\vert}{\vert w_{1}\ldots w_{m_j-n_j}\vert\vert v_1\ldots v_{m_j}\vert}=\frac{e^{\log^{1+\varepsilon}(m_j)-\log^{1+\varepsilon}(m_j-n_j)}}{m_j^{\beta}}.
			\]
Since $\log(1-\delta)\leq \log(1-n_j/m_j)\leq\log(1-c)<0$, observe that 
	\[
	\log^{1+\varepsilon}(m_j)-\log^{1+\varepsilon}(m_j-n_j)-\beta\log(m_j)=-(1+\varepsilon)\log^{\varepsilon}(m_j)\log\left(1-\frac{n_j}{m_j}\right)\left(1+o(1)\right),	
	\]			
which implies
\[
\frac{\vert w_{1}\ldots w_{m_j}\vert}{\vert w_{1}\ldots w_{m_j-n_j}\vert\vert v_1\ldots v_{m_j}\vert}\to\infty\hbox{ as }j\to\infty.
\]			
	Proposition \ref{ctrex2} allows us to conclude.	\\
	
	Now we deal with the case $\varepsilon=1$. In this instance, we will apply Theorem \ref{weightedshiftscounta22}. Our next step is to check the required conditions. Let $1<\gamma<m-1$ be positive integers to be specified later. Since $\frac{v_1\ldots v_n}{w_1\ldots w_n}\to 0$ as $n\to \infty$, we are interested in the convergence of the series 
	$$ \sum_{n\geq 1}\sup_{l\in[m^n,\gamma m^n]\cap\mathbb{N}}\left(\sum_{j=\lfloor\alpha_l\rfloor+1}^{m^{n+1}-l-1}\frac{1}{\vert w_1\ldots w_j\vert}+\sum_{j\geq m^{n+1}}\frac{\vert w_1\ldots w_j\vert}{ \vert  w_1\ldots w_{j-l}\vert\ \vert v_1\ldots v_j\vert}\right)$$
	and 
	$$\sum_{n\geq 1}\sup_{l\in[m^n,\gamma m^n]\cap\mathbb{N}}\left( \sum_{j=\lfloor\alpha_l\rfloor+1}^{m^{n+1}-l-1}\frac{1}{\vert v_1\ldots v_j\vert}+\sum_{j\geq m^{n+1}}\frac{1}{\vert v_1\ldots v_{j-l}\vert}\right)$$
	with $(\alpha_n)$ and $2\leq\gamma<m$ well chosen.\\ 
	On one hand, for all $n$ sufficiently large, $l\in [m^n,\gamma m^n]$ and $j\geq m^{n+1}$, we get
	\[
	\frac{\vert w_{1}\ldots w_{j}\vert}{\vert w_{1}\ldots w_{j-l}\vert\vert v_1\ldots v_{j}\vert}=\frac{e^{\log^{2}(j)-\log^{2}(j-l)}}{j^{\beta}}=e^{-2\log(j)\log(1-\frac{l}{j})-\log^2(1-\frac{l}{j})-\beta\log(j)}.
	\]
	Since, for all $l\in [m^n,\gamma m^n]$ and $j\geq m^{n+1}$,
	\[
	-\log\left(1-\frac{l}{j}\right)\leq -\log\left(1-\frac{\gamma}{m}\right)\hbox{ and }e^{-\log^2(1-\frac{l}{j})}\leq 1,
	\]		
we deduce
\[ 
\frac{\vert w_{1}\ldots w_{j}\vert}{\vert w_{1}\ldots w_{j-l}\vert\vert v_1\ldots v_{j}\vert}\leq e^{-2\log(j)(\beta+\log(1-\frac{\gamma}{m}))}.
\]	
Let us choose $m>\frac{\gamma}{1-e^{-(\beta-1)}}$. We obtain $2(\beta+\log(1-\frac{\gamma}{m}))>2$. Thus we deduce
\[
\sum_{j\geq m^{n+1}}\frac{\vert w_{1}\ldots w_{j}\vert}{\vert w_{1}\ldots w_{j-l}\vert\vert v_1\ldots v_{j}\vert}\leq \sum_{j\geq m^{n+1}}j^{-2}\lesssim m^{-(n+1)}.
\]
The series 	$\sum_{n\geq 1}m^{-(n+1)}$ is convergent. Moreover, for every $l\in[m^n,\gamma m^n]$ large enough, taking $\alpha_l=\log^{\sigma} (l)$ with $\sigma(\beta-1)>1$, 
\[
\sum_{j\geq \lfloor\alpha_l\rfloor+1}\vert w_1\ldots w_j\vert^{-1}=\sum_{j\geq \lfloor\alpha_l\rfloor+1}e^{-\log^2(j)}\lesssim \sum_{j\geq \lfloor\alpha_l\rfloor+1}j^{-\beta}\lesssim 
(\lfloor\alpha_l\rfloor+1)^{-\beta}\leq \left(\frac{1}{n\log (m)}\right)^{\sigma(\beta-1)}
\]
and the series $\sum_{n\geq 1}\frac{1}{n^{2\sigma}\log^{\sigma} (m)}$ is convergent since $\sigma(\beta-1)>1$.\\ 
On the other hand, for all $n$ sufficiently large, $l\in [m^n,\gamma m^n]$ and $j\geq m^{n+1}$, we get
\[
\sum_{j\geq m^{n+1}}\frac{1}{\vert v_1\ldots v_{j-l}\vert}\leq \sum_{j\geq m^{n+1}}\frac{j^{-\beta}}{(1-\frac{\gamma}{m})^{\beta}} 
\lesssim m^{-(\beta-1)(n+1)},
\]
and the series is $\sum_{n\geq 1}m^{-(\beta-1)(n+1)}$ is convergent. Furthermore, it is easy to check that 
\[
\sum_{j\geq \lfloor\alpha_l\rfloor+1}\vert v_1\ldots v_j\vert^{-1}=\sum_{j\geq \lfloor\alpha_l\rfloor+1}j^{-\beta}\lesssim 
(\lfloor\alpha_l\rfloor+1)^{-\beta}\leq \left(\frac{1}{n\log (m)}\right)^{\sigma(\beta-1)}
\]
and the series $\sum_{n\geq 1}\frac{1}{n^{\sigma (\beta-1)}\log^{\sigma} (m)}$ is convergent. The hypotheses of Theorem \ref{weightedshiftscounta22} are satisfied. The conclusion holds. \\

Finally the case $0<\varepsilon<1$ is left to the reader, as it follows similarly from the case $\varepsilon=1$ with simpler computations. 
\end{proof}}

\end{example}

\section{Polynomials of weighted backward shifts}\label{poly} We extend the existence criterion for common frequently hypercyclic vectors from Section \ref{abstheory} to the case of denumerable families of polynomials of backward shift operators.

\subsection{Notations} Let $1\leq p<\infty$. Let $B_{\boldsymbol{w}}$ a weighted frequently hypercyclic backward shift operator on $\ell_p$. We are interested in the polynomials of $B_{\boldsymbol{w}}$. We retain the notation from Section \ref{prelim}. We will also use the fact that, under suitable conditions on the polynomials, there exists an adapted basis related to the backward shift $B_{\boldsymbol{w}}$ such that $P(B_{\boldsymbol{w}})$ can be interpreted as a backward shift operator with respect to this basis (see Lemma 4.1 in \cite{MouMun}). 

\begin{lemma}\label{constructionpoly} Let $1\leq p<\infty$. Let $P=\sum_{i=1}^d\lambda_iz^i$ be a polynomial, with $\lambda_1\ne0$ and $\vert\lambda_1\vert >\sum_{i=2}^d\vert \lambda_i\vert$. Let also  $B_{\boldsymbol{w}}$ be a weighted backward shift on $\ell_p$ that is frequently hypercyclic. 
	There exist polynomials $u_0=e_0$ and $u_k=\sum_{j=0}^{k}\beta_{j,k}e_j,$ 
	such that, for every $k=1,2,\dots$ and for every $l=1,\dots,k,$ we have
	\begin{enumerate}[(i)]
		\item $P(B_{\boldsymbol{w}})(u_{k})=u_{k-1},$
		\item $\beta_{0,k}=0,$ $\beta_{k,k}\ne 0,$  $\vert \beta_{l,k}\vert\leq \left(\displaystyle\frac{1}{\vert\lambda_1\vert-\sum_{j=2}^{d_k}
			\vert\lambda_j\vert}\right)^k\displaystyle\frac{1}{\vert w_1w_2\dots w_l\vert},$
		\item $\Vert u_{k}\Vert_{p}\leq 
		\displaystyle\frac{C_{\boldsymbol{w}}}{(\vert\lambda_1\vert-\sum_{j=2}^d
			\vert\lambda_j\vert)^{k}}$,
	\end{enumerate}
	where $C_{\boldsymbol{w}}:=\left(\sum_{j\geq 1}\frac{1}{\vert w_1\ldots w_j\vert^p}\right)^{1/p}$.
\end{lemma}

\begin{definition} \rm{Let $1\leq p<\infty$. A polynomial $P=\sum_{i=1}^d\lambda_iz^i$ such that $\lambda_1\ne 0$ and $\vert\lambda_1\vert >\sum_{i=2}^d\vert \lambda_i\vert$ is said to be \textit{admissible}. Let also  $B_{\boldsymbol{w}}$ be a weighted backward shift on $\ell_p$ that is frequently hypercyclic. The sequence $(u_j)$ in $c_{00}$ given by Lemma \ref{constructionpoly} is said to be \textit{associated to the operator $P(B_{\boldsymbol{w}})$}.}
\end{definition}

Now let us consider a family of polynomials of weighted backward shifts $(P_k(B_{\boldsymbol{w(k)}}))_{k\in\mathbb{N}}$ where, for every positive integer $k$, $P_k:=\sum_{i=1}^{d_k}\lambda_i^{(k)}z^i$ is an admissible polynomial and $B_{\boldsymbol{w(k)}}$ is a weighted backward shift on $\ell_p$ that is frequently hypercyclic, i.e. 
$$C_{\boldsymbol{w(k)}}:=\left(\sum_{j\geq 1}\frac{1}{\vert w_1(k)\ldots w_j(k)\vert^p}\right)^{1/p}<\infty.$$
We denote by $(\boldsymbol{u^{(k)}})$ the associated sequence $\boldsymbol{u^{(k)}}=(u_j^{(k)})$ to the operator $P_k(B_{\boldsymbol{w(k)}})$ in $c_{00}$ given by Lemma \ref{constructionpoly}. Thus setting $u_j^{(k)}:=\sum_{i=0}^j\beta_{i,j}^{(k)}e_j$, we have, for every $j=1,2,\dots$, and for every $l=1,\dots,j$,
\begin{equation}\label{polyhypofonda}
	\begin{minipage}{0.9\textwidth} 
		\begin{itemize}
			\item $P_k(B_{\boldsymbol{w(k)}})(u_{j})=u_{j-1},$
			\item $\beta_{0,j}^{(k)}=0,$ $\beta_{j,j}^{(k)}\ne 0,$  $\vert \beta_{l,j}^{(k)}\vert\leq \left(\displaystyle\frac{1}{\vert\lambda_1^{(k)}\vert-\sum_{i=2}^{d_k}
				\vert\lambda_i^{(k)}\vert}\right)^j\displaystyle\frac{1}{\vert w_1(k)\ldots w_l(k)\vert},$
			\item $\Vert u_{j}^{(k)}\Vert_{p}\leq 
			\displaystyle\frac{C_{\boldsymbol{w(k)}}}{(\vert\lambda_1^{(k)}\vert-\sum_{i=2}^{d_k}
				\vert\lambda_i^{(k)}\vert)^{j}}.$
		\end{itemize}
	\end{minipage}
\end{equation}
Observe that, for the operator $P_k(B_{\boldsymbol{w(k)}})$, the following crucial property always holds 
$$P_k(B_{\boldsymbol{w(k)}})(\sum_{k\geq 0}X_ku_k)=\sum_{k\geq 0}X_{k+1}u_k\hbox{ and }P_k(B_{\boldsymbol{w(k)}})=B[u^{(k)}].$$ 
From here on, we simplify the notations by setting 
\[ \Pi_k:=P_k(B_{\boldsymbol{w(k)}})\quad\hbox{ and }\quad \rho_k:=\displaystyle\frac{1}{\vert\lambda_1^{(k)}\vert-\sum_{i=2}^{d_k}
	\vert\lambda_i^{(k)}\vert}.
\]
Let us also define the following sequence of operators
$$\begin{array}{rccl} T_n(\Pi_k^n)\ :& E_p[\boldsymbol{u^{(k)}}]&\rightarrow &\ell_p\\&
	\sum_{j\geq 0}X_ju_j^{(k)}&\mapsto &\sum_{j=0}^{n}X_{j+n}u_j^{(k)},\end{array}$$ 
where $E_p[\boldsymbol{u^{(k)}}]=\left\{\sum_{j\geq 0}x_ju_j^{(k)};\ \sum_{k\geq 0}\vert x_j\vert \Vert u_j^{(k)}\Vert_{p}<+\infty\right\}\subset \ell_p.$ \\

\subsection{Common frequent hypercyclicity} Again throughout this section, given a weight sequence $(w_n)$, we denote its partial products by $W_n:=w_1\ldots w_n$, ($n=1,2,\dots$) with the convention $W_0=1$.\\

In the following, we consider  a family $(P_k)_{k\in\mathbb{N}}$, $P_k:=\sum\limits_{i=1}^{d_k}\lambda_i^{(k)}z^i$, of admissible polynomials such that there exists $\delta>0$ so that, for all positive integer $k$, $\vert\lambda_1^{(k)}\vert-\sum_{i=2}^{d_k}\vert\lambda_i^{(k)}\vert\geq 1+\delta$. Thus, using the previous notations, for all $k\geq 1$, 
$$\rho_k\leq \frac{1}{1+\delta}:=\rho.$$
Let $1\leq p<\infty$. Let also $(X_j)$ be a sequence of i.i.d. complex random variables in $L^1$. Let $\boldsymbol{w(k)}=(w_n(k))_{n\in\mathbb{N}}$, $k\in \mathbb{N}$, be countably many weights for which every weighted backward shift $B_{\boldsymbol{w(k)}}$, $k\in \mathbb{N}$, is a continuous operator on the complex Banach space $\ell_p$ and $\inf_{k\in\mathbb{N}} \vert W_n(k)\vert>0$, for all $n\in\mathbb{N}$. We obtain the following statement which gives the existence of a common frequently hypercyclic random vectors for the family of operators $(P_k(B_{\boldsymbol{w(k)}}))_{k\geq 1}$.

\begin{theorem}\label{polyweightedshiftscounta22} Let $1\leq p<\infty$ and $\boldsymbol{w(k)}=(w_n(k))_{n\in\mathbb{N}}$, $k\in \mathbb{N}$, be countably many weights for which every weighted backward shift $B_{\boldsymbol{w(k)}}$, $k\in \mathbb{N}$, is a continuous operator on $\ell_p$. Let $(P_k)_{k\in\mathbb{N}}$, $P_k:=\sum\limits_{i=1}^{d_k}\lambda_i^{(k)}z^i$, be a family of polynomials such that there exists $\delta>0$ so that, for all positive integers $k$, $\vert\lambda_1^{(k)}\vert-\sum_{i=2}^{d_k}\vert\lambda_i^{(k)}\vert\geq 1+\delta$. Let $(X_j)$ be a sequence of i.i.d. complex random variables in $L^1$. Assume that 
	\begin{enumerate}[(i)]
		\item there exists $C\geq 1$ such that, for any $k\in\mathbb{N}$, $\limsup\limits_{n\to\infty}\vert w_n(k)\vert\leq C$;
		%\item for all $n\in\mathbb{N}$, $\inf_{k\in\mathbb{N}} \vert W_n(k)\vert>0$;
		\item the series $\sum_{j\geq 1}\left( \inf_{k\in\mathbb{N}}\vert W_j(k)\vert \right)^{-p}$ is convergent.
	\end{enumerate}	
	Then there exists a positive integer $m\geq 2$ such that the vector $Z=\sum_{j\geq m}X_j u_j^{(\psi\left(\lfloor \frac{\log (j)}{\log (m)}\rfloor\right))}$ is a common frequently hypercyclic random vector for the family $(P_k(B_{\boldsymbol{w(k)}}))_{k\in\mathbb{N}}$ of operators on $\ell_p$, where, for all $k\geq 1$,  $(u_j^{(k)})$ is the associated sequence to the operator $P_k(B_{\boldsymbol{w(k)}})$
\end{theorem}

\begin{proof} Let us define $C_{\boldsymbol{\inf}}:=\left(\sum_{j\geq 1}\frac{1}{\inf_{k\in\mathbb{N}}\left\vert W_{j}(k)\right\vert^p}\right)^{1/p}<\infty$. Notice that the conditions imply that each backward shift $B_{\boldsymbol{w(k)}}$is a frequently hypercyclic operator on $\ell_p$. As usual, we set, for all $k\geq 1$, $C_{\boldsymbol{w(k)}}:=(\sum_{j\geq 1}\vert W_j(k)\vert^{-p})$. 
We define, for all positive integers $k$, 
	$${Z}_k:=\sum_{j\geq 1}X_j u_j^{(k)}.$$
	Let $\gamma\geq 2$ be an integer. Fix a positive integer $m\geq 2\gamma +1$ to be specified later. Let us also consider the random vector 
	$$Z=\sum\limits_{j\geq m} X_ju_{j}^{(\psi(\lfloor\frac{\log (j)}{\log(m)}\rfloor))}=\sum_{n\geq 1}\sum_{j=m^n}^{m^{n+1}-1}X_ju_j^{(\psi(n))}.$$
	The vector $Z$ is thus constructed by concatenating specific blocks from the vectors $(Z_k)_{k\geq 1}$. First, we show that these blocks satisfy the frequent approximation property as stated in Theorem \ref{gdnbre}. Let $k\geq 1$, $h$ be in $c_{00}$ and $\varepsilon>0$. Since the random variables $X_k$ are in $L^p$, Theorem \ref{gdnbre} ensures that there exists a subset $A_k$ of $\mathbb{N}$ with $\rm{dens}(A_k)>0$ such that 
	\begin{equation}\label{dpoly1001}
		\Vert T_{n}(\Pi_k^n)({Z_k})-h\Vert_p<\varepsilon\quad \hbox{almost surely}.
	\end{equation}	
	 For all $n\in R_k$ and for every $l\in [m^n,\gamma m^n]\cap\mathbb{N}$, we get
	\begin{align}\label{d2equpfmain4}
		\begin{split}
			\Pi_k^l(Z)=&\sum_{j=l}^{2l}X_ju_{j-l}^{(k)} +\sum_{j=2l+1}^{m^{n+1}-1}X_ju_{j-l}^{(k)}+
			\sum_{j\geq m^{n+1}}X_j\Pi_k^l\left(u_{j}^{(\psi(\lfloor\frac{\log (j)}{\log(m)}\rfloor))}\right)\\:=&\ \sum_{j=l}^{2l}X_ju_{j-l}^{(k)}+Q_{n,l,k,1}+Q_{n,l,k,2}.
		\end{split}
	\end{align}
	with		
	\begin{align}\label{d2equpfmain4bis}
		\begin{split}
			\Pi_k^l\left(u_{j}^{(\psi(\lfloor\frac{\log (j)}{\log(m)}\rfloor))}\right)=&\sum_{s=1}^{d_k}\lambda_{s}^{(k)}B_{\boldsymbol{w(k)}}^{s+l}\left(\sum_{i=0}^j\beta_{i,j}^{(\psi(\lfloor\frac{\log (j)}{\log(m)}\rfloor))}e_i\right)\\
			=&\sum_{s=1}^{d_k}\sum_{i=s+l}^j\lambda_{s}^{(k)}\beta_{i,j}^{(\psi(\lfloor\frac{\log (j)}{\log(m)}\rfloor))}w_{i-(s+l)+1}(k)\ldots w_{i}(k)e_{i-(s+l)},
		\end{split}
	\end{align}
	We need to estimate the terms $Q_{n,l,k,1}$ and $Q_{n,l,k,2}$. We begin with the term $Q_{n,l,k,2}$: Lemma \ref{estim_var}, properties (\ref{polyhypofonda}) and the assumptions of the theorem lead to, for all $n\in R_k$ and for every $l\in [m^n,\gamma m^n]\cap\mathbb{N}$, almost surely
	\begin{align}
		\begin{split}
			\sum_{j\geq m^{n+1}}\sum_{s=1}^{d_k}\sum_{i=s+l}^j\lambda_{s}^{(k)}&\beta_{i,j}^{(\psi(\lfloor\frac{\log (j)}{\log(m)}\rfloor))}\frac{W_i(k)}{W_{i-(s+l)}(k)}e_{i-(s+l)}\\
				=&
				\sum_{t\geq 0}e_t\left(\sum_{s=1}^{d_k}\lambda_{s}^{(k)}\sum_{j\geq\max(m^{n+1},t+s+l)}X_j\beta_{t+s+l,j}^{(\psi(\lfloor\frac{\log (j)}{\log(m)}\rfloor))}\frac{W_{t+s+l}(k)}{W_t(k)}\right)
		\end{split}		
	\end{align}

The triangle inequality, Lemma \ref{estim_var} and properties (\ref{polyhypofonda}) lead to, for any $k\geq 1$, for all $n\in R_k$ and for any $l\in [m^n,\gamma m^n]\cap\mathbb{N}$,
\begin{align}
	\begin{split}
\Vert Q_{n,l,k,2}\Vert_p^p=&\sum_{t\geq 0}\left\vert\sum_{s=1}^{d_k}\lambda_{s}^{(k)}\sum_{j\geq\max(m^{n+1},t+s+l)}X_j\beta_{t+s+l,j}^{(\psi(\lfloor\frac{\log (j)}{\log(m)}\rfloor))}\frac{W_{t+s+l}(k)}{W_t(k)} \right\vert^p\\
\leq& \sum_{t\geq 0}\left(\sum_{s=1}^{d_k}\vert\lambda_{s}^{(k)}\vert\sum_{j\geq\max(m^{n+1},t+s+l)}\vert X_j\vert \left\vert\beta_{t+s+l,j}^{(\psi(\lfloor\frac{\log (j)}{\log(m)}\rfloor))}\right\vert\frac{\left\vert W_{t+s+l}(k)\right\vert}{\left\vert W_t(k)\right\vert} \right)^p\\
\leq& \sum_{t\geq 0}\left(\sum_{s=1}^{d_k}\vert\lambda_{s}^{(k)}\vert\sum_{j\geq m^{n+1}}\vert X_j\vert \frac{\rho^j}{\left\vert W_{t+s+l}(\psi(\lfloor\frac{\log (j)}{\log(m)}\rfloor))\right\vert} 
\frac{\left\vert W_{t+s+l}(k)\right\vert}{\left\vert W_t(k)\right\vert} \right)^p:=A^p(\omega).
\end{split}
\end{align}

By classical duality on $\ell_p$ spaces, we have, with $q$ such that $\frac{1}{p}+\frac{1}{q}=1$, 
\[
A(\omega)=\sup_{\Vert u(\omega)\Vert_q\leq 1}A(u)(\omega),
\]
with 
\[
A(u)(\omega):=\sum_{t\geq 0}\vert u_t(\omega)\vert\sum_{s=1}^{d_k}\vert\lambda_{s}^{(k)}\vert\sum_{j\geq m^{n+1}}\vert X_j\vert \frac{\rho^j}{\left\vert W_{t+s+l}(\psi(\lfloor\frac{\log (j)}{\log(m)}\rfloor))\right\vert} 
\frac{\left\vert W_{t+s+l}(k)\right\vert}{\left\vert W_t(k)\right\vert}.
\]
By Fubini's Theorem, we get
\[
A(u)(\omega)\leq \sum_{s=1}^{d_k}\vert\lambda_{s}^{(k)}\vert\sum_{j\geq m^{n+1}}\vert X_j\vert\rho^j\sum_{t\geq 0}\vert u_t(\omega)\vert\frac{\left\vert W_{t+s+l}(k)\right\vert}{\left\vert W_{t+s+l}(\psi(\lfloor\frac{\log (j)}{\log(m)}\rfloor))\right\vert\left\vert W_t(k)\right\vert}.
\]

Applying H\"older's inequality to the innermost sum, we obtain the following for sufficiently large $l$ large, recalling that $\Vert u(\omega)\Vert_q\leq 1$,
\begin{align}
	\begin{split}
		A(u)(\omega)\leq&\sum_{s=1}^{d_k}\vert\lambda_{s}^{(k)}\vert\sum_{j\geq m^{n+1}}\vert X_j\vert\rho^j\left(\sum_{t\geq 0}\vert u_t(\omega)\vert^q\right)^{\frac{1}{q}}\left(\sum_{t\geq 0}\frac{\left\vert W_{t+s+l}(k)\right\vert^p}{\left\vert W_{t+s+l}(\psi(\lfloor\frac{\log (j)}{\log(m)}\rfloor))\right\vert^p\left\vert W_t(k)\right\vert^p}\right)^{\frac{1}{p}}		
	\\
	\leq&  \sum_{s=1}^{d_k}\vert\lambda_{s}^{(k)}\vert\sum_{j\geq m^{n+1}}\vert X_j\vert\rho^j\left(\sum_{t\geq 0}\frac{\left\vert W_{t+s+l}(k)\right\vert^p}{\left\vert W_{t+s+l}(\psi(\lfloor\frac{\log (j)}{\log(m)}\rfloor))\right\vert^p\left\vert W_t(k)\right\vert^p}\right)^{\frac{1}{p}}
	\\
		\leq & \sum_{s=1}^{d_k}\vert\lambda_{s}^{(k)}\vert\sum_{j\geq m^{n+1}}\vert X_j\vert\rho^jR_k(s,l),
		\end{split}
\end{align}
with 
\[
R_k(s,l):=\left(\sum_{t\geq 0}\frac{\left\vert W_{t+s+l}(k)\right\vert^p}{\left\vert W_t(k)\right\vert^p}\frac{1}{\inf_i\left\vert W_{t+s+l}(i)\right\vert^p}\right)^{\frac{1}{p}}.
\]
Let $\eta>0$ be fixed and arbitrarily. Then, for all $k\geq 1$, there exist $m_k$ such that, for all $n\geq m_k+1$, 
\[
\vert w_n(k)\vert\leq (1+\eta)C.
\]
Moreover, we set
\[
U_k:=\max(1,\sup_{1\leq n\leq m_k} \vert w_n(k)\vert).
\]
Therefore, we obtain the following estimate, for all positive integer $n\in R_k$ and all $l\in[m^n,\gamma m^n]\cap\mathbb{N}$ large enough,

\begin{align}
	\begin{split}
		\left(R_k(s,l)\right)^p\leq&\ \sum_{t\leq m_k}\frac{U_k^{p(m_k-t)}((1+\eta)C)^{p(t+s+l-m_k)}}{\inf_i\left\vert W_{t+s+l}(i)\right\vert^p}+ 
		\sum_{t\geq m_k+1}\frac{((1+\eta)C)^{p(s+l)}}{\inf_i\left\vert W_{t+s+l}(i)\right\vert^p}
		\\
		\leq& \ U_k^{pm_k}((1+\eta)C)^{p(s+l)}\sum_{t\geq 0}\frac{1}{\inf_i\left\vert W_{t+s+l}(i)\right\vert^p}
		\\
		\leq&\ U_k^{pm_k}((1+\eta)C)^{p(s+l)}C_{\boldsymbol{\inf}}^p.
	\end{split}
\end{align}

Then we infer, for all positive integer $n\in R_k$ and all $l\in[m^n,\gamma m^n]\cap\mathbb{N}$ large enough,

\begin{align}
	\begin{split}
		A(u)(\omega)\leq&\ U_k^{m_k}((1+\eta)C)^{l}C_{\boldsymbol{\inf}}
		\sum_{s=1}^{d_k}\vert\lambda_{s}^{(k)}\vert((1+\eta)C)^{s}\sum_{j\geq m^{n+1}}\vert X_j\vert\rho^j
		\\
		\lesssim&\ ((1+\eta)C)^{l} \sum_{j\geq m^{n+1}}\vert X_j\vert\rho^j.
	\end{split}
\end{align}

Therefore, we obtain, for all positive integer $n\in R_k$ and all $l\in[m^n,\gamma m^n]\cap\mathbb{N}$ large enough,
\[
A(\omega)\lesssim ((1+\eta)C)^{l} \sum_{j\geq m^{n+1}}\vert X_j\vert\rho^j.
\]

We derive, for all positive integer $n\in R_k$ and all $l\in[m^n,\gamma m^n]\cap\mathbb{N}$ large enough,
\[
\mathbb{E}[\Vert Q_{n,l,k,2}\Vert_p]\lesssim \left(((1+\eta)C)^{\gamma}\rho^{m}\right)^{m^n}
\]
Let us choose $m>\frac{\gamma \log ((1+\eta)C)}{-\log(\rho)}$.Thus Markov's inequality and Borel-Cantelli Lemma ensure that, for $n$ sufficiently large,
\begin{equation} \label{ddabstracteqq2}
	\Vert Q_{n,l,k,2}\Vert_p<\varepsilon \quad \text{almost surely.} 
\end{equation}

Moreover we estimate the term $Q_{n,l,k,1}$. Using properties (\ref{polyhypofonda}), we get, for all $n\in R_k$ and for every $l\in [m^n,\gamma m^n]\cap\mathbb{N}$, 
\begin{equation}
\begin{array}{rcl}\displaystyle\mathbb{E}\left[\Vert Q_{n,l,k,1}\Vert_p\right]&\lesssim& 	\displaystyle\sum_{j=2l+1}^{m^{n+1}-1}\Vert u_{j-l}^{(k)}\Vert_p\\&\lesssim& 
\displaystyle\sum_{j=2l+1}^{m^{n+1}-1}C_{\boldsymbol{w(k)}}\rho_k^{j-l}\\
&\lesssim &\displaystyle\rho_k^{m^n} \frac{\rho_k}{1-\rho_k}\end{array}
\end{equation}
Observe that $\sum_{n\geq 1}\rho_k^{m^n} <\infty$, thus Markov's inequality and Borel-Cantelli Lemma ensure that, for $n$ sufficiently large,
\begin{equation} \label{dabstracteqq2}
	\Vert Q_{n,l,k,1}\Vert_p<\varepsilon \quad \text{almost surely.} 
\end{equation}
Finally, note that, for all $n$ large enough, $l\in [m^n,\gamma m^n]\cap\mathbb{N}$, 
\begin{equation}\label{ddabstracteqq2222}
	\sum_{j=l}^{2l}X_ju_{j-l}^{(k)}-h= T_{l}(\Pi_k^l)({Z}_k)-h.
\end{equation}
Combining (\ref{dpoly1001}) with (\ref{ddabstracteqq2}), (\ref{dabstracteqq2}), and (\ref{ddabstracteqq2222}), we obtain that there exist a positive integer $N_k$ and a subset $A_k$ of $\mathbb{N}$ with $\hbox{dens}(A_k)>0$ such that for all $n\in A_k\cap \bigcup\limits_{n\in R_k;\atop n\geq N_k} [m^n,\ldots,\gamma m^{n}]\cap\mathbb{N}$ 
\[
\Vert T_{l}(\Pi_k^l)({Z})-h\Vert_p<3\varepsilon \quad\hbox{almost surely.}
\]
Since Lemma \ref{lemma_combination_density} guarantees that $\underline{\hbox{dens}}(A_k\cap \bigcup\limits_{n\in R_k;\atop n\geq N_k} [m^n,\ldots,\gamma m^{n}]\cap\mathbb{N}) >0$, Lemma \ref{lemmefondamproba} allows us to conclude that the random vector $Z$ is almost surely a common frequently hypercyclic vector for the family $(P_k(B_{\boldsymbol{w(k)}}))_{k\in \mathbb{N}}$.
\\

\end{proof}

\begin{remark}{\rm 
	\begin{enumerate}
		\item Observe that Theorem \ref{polyweightedshiftscounta22} shows that, due to the nature of polynomials, polynomials of backward shift operators introduce an additional geometric decay. This allows us to relax the assumption given in Theorem \ref{weightedshiftscounta22} that weight products must be, in a sense, comparable, in order to obtain common frequently hypercyclic vectors for countable families of polynomials of weighted backward shifts.
		\item In the case of a finite number of operators $P_k(B_{\boldsymbol{w(k)}})$, $k=1,\dots,N$, the proof of Theorem \ref{polyweightedshiftscounta22} is somewhat simpler. Indeed, for the estimation of the term $Q_{n,l,k,2}$, it suffices to notice that, for all $l\in [m^n,\gamma m^n]\cap\mathbb{N}$, keeping the same notations, 
		\[
		\begin{split}
			\mathbb{E}[\Vert Q_{n,l,k,2}\Vert_p]\leq&\ \sum_{j\geq m^{n+1}}\mathbb{E}[X_j]\Vert\Pi_k^l(u_j^{(\psi(\lfloor\frac{\log (j)}{\log(m)}\rfloor))})\Vert_p\\
		\lesssim&\ M^l \sum_{j\geq m^{n+1}}\Vert u_j^{(\psi(\lfloor\frac{\log (j)}{\log(m)}\rfloor))}\Vert_p\\
		\lesssim&\ M^l \sum_{j\geq m^{n+1}}\rho^{j}\\
		\lesssim&\ (M^{\gamma}\rho^m)^{m^n},
		\end{split}
		\]
where $M:=\max\limits_{k=1,\dots,N}\sum\limits_{i=1}^{d_k}\vert \lambda_i^{(k)} \vert (\sup_{n\in\mathbb{N}}\vert w_n(k)\vert)^i$. The estimation of the term $\Vert Q_{n,l,k,2}\Vert_p$ follows using Borel-Cantelli Lemma with a good choice of $m$. 
	\end{enumerate}	}
\end{remark}

We can also deduce the following corollaries. 

\begin{corollary}\label{dweightedshiftscounta23} Let $1\leq p<\infty$. Let $\beta>1/p$ and $(\beta_k)_{k\in\mathbb{N}}$ a sequence of real numbers so that, for all $k\geq 1$, $\beta_k\geq\beta$ . Let $(\boldsymbol{w(k)})_{k\in\mathbb{N}}$ a family of weights sequence satisfying: there exists $A\geq 1$ , such that for all $k\in\mathbb{N}$, $A^{-1}n^{\beta_k}\leq \vert w_1(k)\ldots w_n(k)\vert\leq A n^{\beta_k}$. Let also $(P_k)_{k\in\mathbb{N}}$, $P_k=\sum\limits_{i=1}^{d_k}\lambda_i^{(k)}z^i$, be a family of polynomials such that there exists $\delta>0$ so that, for all positive integers $k$, $\vert\lambda_1^{(k)}\vert-\sum_{i=2}^{d_k}\vert\lambda_i^{(k)}\vert\geq 1+\delta$. Then there exists a common frequently hypercyclic random vector for the family $(P_k(B_{\boldsymbol{w(k)}}))_{k\in \mathbb{N}}$ on $\ell_p$.
\end{corollary}

\begin{proof} The result follows by a direct application of Theorem \ref{polyweightedshiftscounta22}, because the assumptions ensure that every weighted backward shift $B_{\boldsymbol{w(k)}}$, $k\in \mathbb{N}$, is a continuous frequently hypercyclic operator on $\ell_p$, $\limsup\limits_{n\to\infty}\vert w_n(k)\vert<\infty $ and 
	\[
	\sum_{j\geq 1}\frac{1}{\left(\inf_i\vert W_j(i)\vert \right)^{p}}\leq \sum_{j\geq 1}\frac{1}{j^{p\beta}}<\infty.
	\]
	
\end{proof}

Similarly, we obtain the following statement. 

\begin{corollary}\label{dweightedshiftscounta24} Let $1\leq p<\infty$. Let $(a_k)_{k\in\mathbb{N}}$ be a sequence of complex numbers and $(\boldsymbol{w(k)})_{k\in\mathbb{N}}$ be a family of weights sequence such that, for every $k,n\geq 1$, $A^{-1}\vert a_k\vert^n\leq \vert w_1(k)\ldots w_n(k)\vert\leq A\vert a_k\vert^n$ with $A\geq 1$. Let also $(P_k)_{k\in\mathbb{N}}$, $P_k=\sum\limits_{i=1}^{d_k}\lambda_i^{(k)}z^i$, be a family of polynomials such that there exists $\delta>0$ so that, for all positive integers $k$, $\vert\lambda_1^{(k)}\vert-\sum_{i=2}^{d_k}\vert\lambda_i^{(k)}\vert\geq 1+\delta$. If there exist $1<a<b$ such that for every $k\geq 1$, $a\leq \vert a_k\vert\leq b$, then there exists a common frequently hypercyclic random vector for the family $(P_k(B_{\boldsymbol{w(k)}}))_{k\in \mathbb{N}}$ on $\ell_p$.
\end{corollary}

\vskip3mm

\begin{remark} \label{rem5}\rm{
		\begin{enumerate}
			\item Theorem \ref{polyweightedshiftscounta22} also allows us to recover the fact that if $\Lambda$ is a countable, relatively compact set in $(1,\infty)$, then the family $\{\lambda B;\ \vert \lambda\vert\in\Lambda\}:=\{\lambda_kB;\ k\geq 1,\ \vert\lambda_k\vert\in\Lambda\}$ has a common frequently hypercyclic vector. Indeed, to see this, let $1<a<b$ such that $\Lambda\subset [a,b]$. We then consider the backward weighted shift operator $\sqrt{a}B$, which is clearly frequently hypercyclic. We then set the polynomials $P_k(z)=\frac{\lambda_k}{\sqrt{a}}z$. Note that the aforementioned collection coincides with the family $(P_k(\sqrt{a}B))_{k\geq 1}$ and that these polynomials $P_k$ satisfy the conditions of Theorem \ref{polyweightedshiftscounta22}. Moreover observe that, for all $k\geq 1$, $\frac{\lambda_k}{\sqrt{a}}\geq\sqrt{a}>1$. By applying Corollary \ref{dweightedshiftscounta24}, the result holds: the family $(P_k(\sqrt{a}B))_{k\geq 1}$ shares a common frequently hypercyclic vector.
	\item \label{rem51} Furthermore Theorem \ref{polyweightedshiftscounta22} is also applicable to polynomials of the backward weighted shift operators even if the backward weighted shift operators fail to be frequently hypercyclic. For example, let us consider the operators $5B-6B^2$ and $2B-B^4$. Set $B_1:=\frac{5}{2}B$, $B_2:=\frac{3}{2}B$, $P_1(z):=2z-\frac{24}{25}z^2$ and $P_2(z):=\frac{4}{3}z-\frac{2^4}{3^4}z^4$. Observe that the operators $B_1$ and $B_2$ are frequently hypercyclic weighted backward shift operators and that 
\[
P_1(B_1)=5B-6B^2\hbox{ and }P_2(B_2)=2B-B^4.
\]
Moreover $P_1$ and $P_2$ are admissible polynomials and an easily computation gives
\[
2-\frac{24}{25}>\frac{11}{10}\hbox{ and }\frac{4}{3}-\frac{2^4}{3^4}>\frac{11}{10}.
\] 
Thus let us fix $\rho:=\frac{10}{11}$. Corollary \ref{dweightedshiftscounta24} ensures that the operators $5B-6B^2$ and $2B-B^4$ have a common frequently hypercyclic vector.
	\end{enumerate}
}
\end{remark} 

Remark \ref {rem5} (\ref{rem51}) suggests stating a more general result on the common frequent hypercyclicity of a family of polynomials of a given weighted backward shift, even if this shift is not frequently hypercyclic. Indeed, the dynamics of such operators appear to be driven by the polynomials and specifically by the coefficients of the first-degree term. Actually we can state the following general result, which is an extension of \cite[Corollary 2.24]{CEMM}. First of all, let us introduce, for a bounded weighted backward shift $B_{\boldsymbol{w}}$ on $\ell_p$, the quantity 
\[ 
r_{p,\boldsymbol{w}}:=\sup\{\vert\lambda\vert:\ \lambda\in\sigma_p(B_{\boldsymbol{w}})\},
\]
where $\sigma_p(B_{\boldsymbol{w}})$ denotes the point spectrum of $B_{\boldsymbol{w}}$. The following equality is well-known (see \cite{Shield})
\[
r_{p,\boldsymbol{w}}^{-1}=\limsup\limits_{n\to\infty}\vert w_1\ldots w_n\vert^{-1/n}.
\] 

\begin{corollary}\label{gen224CEMM} Let $1\leq p<\infty$, $\boldsymbol{w}=(w_n)_{n\in\mathbb{N}}$ be a bounded weight sequence and $B_{\boldsymbol{w}}$ the associated bounded weighted backward shift on $\ell_p$. Let also $(P_k)_{k\in\mathbb{N}}$, $P_k:=\sum\limits_{i=1}^{d_k}\lambda_i^{(k)}z^i$, be a family of polynomials and $a,b>0$ such that, for all $k\geq 1$, $r_{p,\boldsymbol{w}}^{-1}<a\leq \vert\lambda_1^{(k)}\vert\leq b$. If there exist a sequence $(r_k)_{k\geq 1}$ of real numbers and $\delta>0$ such that, for all $k\geq 1$, 
	\[r_k\in [a,\vert\lambda_1^{(k)}\vert)\quad\hbox{ and }\quad\frac{\vert\lambda_1^{(k)}\vert r_k^{d_k-1}-\sum_{i=2}^{d_k}\vert\lambda_i^{(k)}\vert r_k^{d_k-i}}{r_k^{d_k}}\geq 1+\delta,
	\]
	then $\bigcap_{k\geq 1}FHC(P_k(B_{\boldsymbol{w}}))\ne\emptyset.$
\end{corollary}

\begin{proof} First observe that, for all $k\geq 1$, the operator $r_kB_{\boldsymbol{w}}$ is frequently hypercyclic and $P_k(B_{\boldsymbol{w}})=\widetilde{P}_k(r_kB_{\boldsymbol{w}})$ where 
	\[
	\widetilde{P}_k(z):=\sum\limits_{i=1}^{d_k}\frac{\lambda_i^{(k)}}{r_k^i}z^i.
	\]
	The assumptions of the theorem entail that
	\[
	\frac{\vert\lambda_1^{(k)}\vert}{r_k}-\sum_{i=2}^{d_k}\frac{\vert\lambda_i^{(k)}\vert}{r_k^i}=\frac{\vert\lambda_1^{(k)}\vert r_k^{d_k-1}-\sum_{i=2}^{d_k}\vert\lambda_i^{(k)}\vert r_k^{d_k-i}}{r_k^{d_k}}\geq 1+\delta.
	\]	 
Then it is easy to check that 
\[
\limsup\limits_{n\to\infty}(\vert r_k w_n\vert)\leq b\sup_{n\in\mathbb{N}}\vert w_n\vert <\infty \quad\hbox{ and }\quad\inf_{k\in\mathbb{N}}\vert r_k^n w_1\dots w_n\vert\geq a^n r_{p,w}^n>0
\]	
which implies, since $a r_{p,w}>1$, $\sum_{n\geq 1}	\inf_{k\in\mathbb{N}}\vert r_k^n w_1\dots w_n\vert^{-p}<\infty$. Thus Theorem \ref{polyweightedshiftscounta22} gives the result.
\end{proof}

\begin{remark}
	\rm{\begin{enumerate}
			\item Let $B_{\boldsymbol{w}}$ be a bounded weighted backward shift on $\ell_p$. Observe that with the choice, for all $k\geq 1$, $P_k(z)= \lambda_kz$, Corollary \ref{gen224CEMM} recovers Corollary 2.24 (1) of \cite{CEMM}, i.e. $\bigcap_{\lambda\in\Lambda}FHC(\lambda B_{\boldsymbol{w}})$ is non-empty whenever $\Lambda$ is a countable relatively compact subset of $(r_{p,\boldsymbol{w}}^{-1},\infty)$.
		\item Similarly to Corollary \ref{gen224CEMM}, we could establish conditions for a family $P_k(B_{\boldsymbol{w(k)}})_{k\geq 1}$ to possess a common frequently hypercyclic vector, even in cases where the shifts $B_{\boldsymbol{w(k)}}$ are not frequently hypercyclic.
		\item The proof method of Theorem 1 allows us to extend the result to entire functions, provided that suitable conditions on the Taylor coefficients of these functions are satisfied.
	\end{enumerate}}
\end{remark}

\end{document}